\documentclass[fleqn]{article}
\usepackage{amsmath} 
\usepackage{amssymb} 
\usepackage{amsfonts}
\usepackage{amsthm}
\usepackage{color}
\usepackage{multirow}
\usepackage{verbatim}
\usepackage[dvips]{graphicx}

\theoremstyle{definition}
\newtheorem{theorem}{Theorem}[section]
\newtheorem{definition}{Definition}[section]

\newtheorem{corollary}{Corollary}[section]
\newtheorem{proposition}{Proposition}[section]
\newtheorem{lemma}{Lemma}[section]

\begin{document}
\title{Accuracy of Maximum Likelihood Parameter Estimators for Heston volatility SDE}
\author{Robert Azencott\\ Professor of Mathematics \\ University of Houston, Houston, TX 77204\\ 
Emeritus Professor, Ecole Normale Superieure, France \\
Email: razencot@math.uh.edu \\
\and \\ Yutheeka Gadhyan \\ University of Houston, Houston, TX 77204\\
Email: yutheeka@gmail.com}
\date{}
\maketitle
\begin{abstract}
We study approximate maximum likelihood estimators (MLEs) for the parameters of the widely used Heston stock and volatility stochastic differential equations (SDEs). We compute explicit closed form estimators maximizing the discretized log-likelihood of $N$ observations recorded at times $T,2T, \ldots, NT$. We study the asymptotic bias of these parameter estimators first for $T$ fixed and $N \to \infty$, as well as when the global observation time $S= NT \to \infty$ and $T = S/N \to 0$. We identify two explicit key functions of the parameters which control the type of asymptotic distribution of these estimators, and we analyze the dichotomy between asymptotic normality and attraction by stable like distributions with heavy tails. \\ 
We present two examples of model fitting for Heston SDEs, one for daily data and one for intraday data, with moderate values of $N$.
\end{abstract}
\newpage
\section{Stochastic dynamics for prices and volatilities }
Stochastic differential equations (SDEs) have often been used to model the joint dynamics of assets price and volatility. Autonomous stochastic volatility models have for instance been studied in \cite{ghysels14renault, barndorff2002some, melino1990pricing, melino1991pricing}. We focus here on the widely used Heston model \cite{heston1993closed}, where asset price and squared volatility are jointly driven by a pair of SDEs with non constant coefficients depending on five parameters. This model is known to be amenable to option price computation \cite{heston1993closed, carr1999option}. \\
Volatilities are usually not directly observable and are then indirectly evaluated from available price data and/or option prices (see \cite{chernov1999estimation, broto2004estimation, ait2007maximum} and references therein). Commonly used volatility estimators are derived either from ``realized volatility" directly computed from stock price data sequences or from the Black-Scholes ``implied volatility" computed by ``inverting" observed option prices. 
\paragraph{}
Here we do not attempt to take account of the actual technique used to generate volatility data, since we will revisit this key question in a companion paper (see \cite{azencott1option}). So we deliberately assume that the ``true" values of asset price and volatility are actually available at $N$ times $T, 2T, \ldots, NT$. The sub-sampling time $T$ is either known or user selected, as is often the case for SDEs model fitting to daily data. We derive explicit expressions for discretized maximum likelihood estimators of the five parameters of the Heston joint SDEs, enabling efficient and very fast numerical parameter estimation.\\ 
We first compute the deterministic asymptotic bias of these parameter estimators for $T$ fixed and $N \to \infty$ and show that these asymptotic biases tend to $0$ when $T \to 0$. 
We prove also the asymptotic consistency of these estimators when the global observation time $S = NT \to \infty$ with sub-sampling time $T = S/N$ tending to 0. \\
We show that for these parameter estimators, the errors of estimation have distributions which remain invariant under arbitrary volatility rescaling and any linear time change. This reduces the numerical study of estimation error distributions for volatility parameters to simulations of a special one-parameter family of Heston volatility SDEs.\\
We identify two explicit key functions of the parameters which control the type of asymptotic distribution of parameter estimators, and we analyze the dichotomy between asymptotic normality and attraction by stable like distributions with heavy tails. 
We illustrate our results by fitting Heston SDEs to 252 joint daily observations of the S\&P 500 and VIX indices, as well as to 5,000 minute by minute intra-day data for the Credit Agricole stock price and its approximate volatility. For these empirical Heston models we implement a numerical study of the accuracy of our parameters estimators for moderate $N$, since realistic values of $N$ typically range from 250 to 750 for daily stock data, and from 2,500 to 25,000 for intra-day liquid stocks. \\
In a companion paper, we study the option pricing sensitivity to parametric estimation errors generated by Heston SDEs fitting to price and volatility data (see \cite{azencott1option}).\\
\section{The Heston joint SDEs}\label{Model}
Let $X_t$ be the asset price at time $t \geq \; 0 $. The volatility $\sqrt{Y_t}$ of $X_t$ is the instantaneous relative volatility of the returns process $dX_t$, classically defined by 
$\; Y_t \; dt \; = \; var (dX_t/X_t)$.\\
In the Heston model \cite{heston1993closed}, $\{X_t,Y_t\}$ is a progressively measurable stochastic process defined on a probability space $(\Omega, \mathcal{F}_t, P)$, endowed with an increasing filtration $\mathcal{F}_t$, and is driven by the following joint SDEs,
\begin{eqnarray}
dX_t &=& \mu X_t dt + \sqrt{Y_t} X_t d Z_t, \label{stock}\\
dY_t &=& \kappa(\theta - Y_t)dt + \gamma\sqrt{Y_t}d B_t. \label{vol} 
\end{eqnarray}
Here $Z_t$ and $B_t$ are standard Brownian motions, adapted to the filtration $\mathcal{F}_t$, and have constant instantaneous correlation $\rho $, so that $E[dZ_t dB_t] = \rho dt$. The parameters $\kappa,\theta, \gamma$ are positive and $\mu$ is the constant mean return rate of the asset price. \\
The autonomous volatility SDE \eqref{vol} driving $Y_t$ is parametrized by the positive numbers $\; \kappa,\theta,\gamma \;$. This process is identical to Feller's ``mean reverting" process \cite{feller1951two} originally used in \cite{cox1985theory} to model short-term interest rates. The classical constraint $\; 2\kappa\theta > \gamma^2 $ introduced in \cite{feller1951two} ensures that any solution $Y_t$ of the SDE \eqref{vol} starting at $Y_0 =y > 0$ remains almost surely positive for all $t$. The volatility process then has then the``mean reversion" property $\; \lim_{t \to \infty} E(Y_t) = \theta$.
\section{Parameter Estimation for Heston joint SDEs}\label{ParEst}
\subsection{Parameter estimation outline}\label{Eststrat}
%%
% % %
Well known methods are available to estimate volatility from market data ( see for instance \cite{genon1994estimation, dohnal1987estimating, dacunha1986estimation} ). But here we will first consider the ideal situation where ``true" price and volatility data are available. We are given $N$ joint observations of the asset price $X_{nT}$ and its squared volatility $Y_{nT}$, with $1\leq n \leq N$, where $T> 0$ is a fixed \textit{ sub-sampling time}. \\
In practice $T$ is user selected or determined by the data context. So we study the situation where $T$ is known and fixed. The unknown parameters $ [ \; \kappa,\theta,\gamma, \rho \; ]$ are constrained to belong to the domain $Dom \subset \mathbb{R}^4$ defined by 
\begin{equation} \label{ParDom}
\kappa > 0, \quad \theta > 0 , \quad \gamma > 0 , \quad 2\kappa\theta -\gamma^2 >0 , \quad |\rho| <1.
\end{equation} 
There is no constraint on the asset price drift $\mu$, which has no actual impact on option pricing equations and is often ignored in data modelling.\\
The joint log-likelihood of Heston SDEs is not easily tractable numerically, so we deliberately decouple the estimations of $ (\kappa,\theta,\gamma)$ and of $(\mu,\rho)$. Our estimators $(\hat \kappa,\hat \theta,\hat \gamma)$ are thus derived only from the sub-sampled volatility data, by maximizing a time discretization of the log-likelihood, a natural approach for small sub-sampling time $T$. \\
We then discretize the two Heston SDEs to estimate the drift $\mu$ and the discrete Brownian increments $\Delta Z_{nT}$ and $\Delta B_{nT} $, before computing their empirical correlation coefficient $\hat \rho$.
\subsection{Change of scale in the volatility SDE} \label{scale}
Practical estimates of the volatility are typically annualized, for instance when they are based on filtered variance of the rate of returns for daily data. Annualization replaces $Y_t$ by the rescaled value $ R_t = A Y_t$ where the scaling factor $A$ is known. 
The pair $(X_t, R_t)$ is then driven by the SDEs 
\begin{equation*}
dX_t = \mu X_t dt + \frac{1}{\sqrt{A}} \sqrt{R_t} X_t d Z_t \quad \quad \text{and} \quad \quad
dR_t = \bar{\kappa}\;( \bar{\theta} - R_t) dt + \bar{\gamma} \sqrt{R_t}d B_t \quad \quad 
\end{equation*} 
where the new parameters of the autonomous SDE verified by $R_t$ are given by
\begin{equation} \label{barchange}
\bar{\kappa} = \kappa \quad \bar{\theta} = A \theta \quad \bar{\gamma} = \sqrt{A}\gamma 
\end{equation}
For most assets, annualized volatilities $\sqrt{R_t}$ are inferior to $100 \%$ in stable markets. Since $ \bar{\theta} = A \theta$ is the long run limit of $E(R_t)$, one then usually has $ A \theta <1 $ with a rescaling factor $A > 1$. Thus in concrete data modelling the mean reversion parameter $\theta$ will typically be within $[0,1]$. 
\subsection{Change of time in the volatility SDE} \label{time}
Let $(X_t,Y_t)$ be driven by the joint Heston SDEs \eqref{stock} \eqref{vol}.
Fix a rescaling factor $\sigma > 0$ and consider the \textit{time change} $\; t \rightarrow s = t/\sigma \;$, which defines new processes $U_s = X_{s \; \sigma}$ and $V_s = Y_{s \; \sigma}$ for all $s > 0$. Then there are two new standard Brownian motions, which we still denote $Z_s, B_s$, with instantaneous correlation $\rho$, such that $(U_s, V_s)$ are driven by the joint SDEs
\begin{equation}\label{changestock}
dU_s = \sigma \mu U_s ds + \sqrt{\sigma} \sqrt{V_s}U_s dZ_s
\end{equation}
\begin{equation}\label{paruvw}
dV_s = (u - v V_s)ds + \sqrt{2w} \sqrt{V_s} dB_s
\end{equation}
where the new parameters $(u,v,w)$ are linked to $ (\kappa, \theta, \gamma^2) $ by 
\begin{equation} \label{uvw2ktg}
u = \sigma \kappa \theta \; , \; v = \sigma \kappa \; , \; w = \sigma \gamma^2 / 2 \quad \text{and} \quad 
\kappa = v/\sigma \; , \; \theta = u/v \; , \; \gamma^2 = 2w/\sigma \quad 
\end{equation}
The parameters constraints \eqref{ParDom} are equivalent to forcing 
$(u, v, w)$ to belong to the {\em convex cone} $\mathcal{S} $ defined by
\begin{equation}\label{dom}
\mathcal{S} = \{(u, v,w) \in \mathbb{R}^3 \; | \; u > w > 0 \;\; \text{and } \;\; v> 0 \}
\end{equation}
To check these elementary statements, simply note that the change of time 
$s = t/\sigma$ replaces the original Brownian motions $Z_t , B_t$ by the new Brownian motions $\; \tilde{Z}_s = Z_{s \; \sigma} / \sqrt{\sigma} \;$ and $\; \tilde{B}_s = B_{s \; \sigma} / \sqrt{\sigma} \;$.
\section{Estimators of volatility parameters}\label{MLEest}
Given a known sub-sampling time $T$ and $N$ observations $Y_{n T}$ of the squared volatility, we perform the preceding change of time $V_s = Y_{s T}$ with rescaling factor $\sigma = T$, so that $V_s$ is driven by SDE 
\eqref{paruvw}. We now seek to estimate the new parameters $u, v, w$ given the $N$ data $V_n = Y_{n T} $, under the constraints \eqref{dom}.
\paragraph{Discretization of volatility dynamics :}
Euler discretization replaces the SDE \eqref{paruvw} by the recurrence relation
\begin{equation}\label{Euler}
V_{n+1} \approx V_n + u - v V_n + \sqrt{2w}\sqrt{V_n} \Delta B_n \quad \text{where} \;\; 
\Delta B_n = B_{n+1} - B_n \quad 
\end{equation}
By \cite{higham2005convergence}, when $T \to 0$ and $N \to \infty$ with $NT$ \textit{bounded}, this Euler discretization provides an asymptotically consistent approximation of the continuous process $\{ \; Y_{nT} \; \}$ . 
\paragraph{Discretized log-likelihood : } 
Due to \eqref{Euler}, the variables $Q_n$ defined by 
\[
Q_n = ( \Delta V_n - u +v V_n ) / \sqrt{V_n} \; \approx \; \sqrt{2w} \Delta B_n \quad \text{where} \;\; 
\Delta V_n = V_{n+1} - V_n \quad \quad 
\]
are approximately independent and Gaussian, with mean $0$ and variance $2w \;$.\\ 
The discretized log-likelihood $L_N$ of $V_1,\ldots,V_N$ is given by, up to a constant,
\[
\frac{2}{N} L_N \; = \; -\log 2\pi - \log(2w) - \frac{S_N}{2w} \quad \text{where} \quad 
S_N = \frac{1}{N}\sum_{n=0}^{N-1} Q^{2}_n \quad \quad 
\]
As easily seen, $S_N$ is a positive quadratic function of $u,v,w$ given by
\[
S_N = a + bu +cv +\frac{1}{2}du^2 -2uv + \frac{1}{2}fv^2
\]
% % 
%
where $\; \mathcal{A}_N = ( a, b, c, d, f )= ( a_N , b_N , c_N , d_N , f_N ) \;$ is the vector of sufficient statistics given by 
\begin{eqnarray}\label{coeff1}
a_N = \frac{1}{N} \sum_{n=0}^{N-1} \frac{(V_{n+1} -V_n)^2}{V_n} \; ; \quad 
b_N = - \frac{2}{N} \sum_{n=0}^{N-1} \frac{V_{n+1} -V_n}{V_n} \\
c_N = \frac{2}{N} (V_N - V_0) \; ; \quad 
d_N = \frac{2}{N} \sum_{n=0}^{N-1} \frac{1}{V_n} \; ;\quad 
f_N = \frac{2}{N}\sum_{n=0}^{N-1} V_n
\end{eqnarray} 
In these formulas, we always set $V_n= Y_{nT}$. We then almost surely have
\begin{equation} \label{discriminant} 
a>0 \;, \;\; d > 0 \;, \;\; f>0 \;,\quad df - 4 >0 \;, \quad 2af - c^2 > 0 \;, \quad d+f -4 > 0
\end{equation} 
Indeed, Cauchy-Schwarz inequality easily implies the positivity of $2af - c^2$ and $df-4$. One then has 
$\; (d+f)^2 \geq 4df > 16 \;$ and hence $\; d+f > 4 \;$.
\paragraph{Log-likelihood maximization : }
To compute an approximate maximum likelihood estimator $\hat h$ of $p = (u,v,w) $, we minimize over all points $p$ in the cone $\mathcal{S}$ the function 
\begin{equation*}
L(p) = L(u,v,w) = \log (2w) + \frac{1 }{2w} \; [ \; a + bu +cv +\frac{1}{2}du^2 -2uv + \frac{1}{2}fv^2 \; ] \quad \quad
\end{equation*}
We have $L(p) \to + \infty$ when $p$ tends to $0$ or $\infty$ in $\mathcal{S}$, so the minimum of $L$ on 
$ \mathcal{S} $ is actually reached on the closure $\overline{\mathcal{S}} $ of $ \mathcal{S}$. Since $L$ is strictly convex in $(u,v)$ for fixed $w$, and in $w$ for fixed $\; u,v \;$, the minimum of $L$ on the closed convex cone $\overline{\mathcal{S}} $ is reached at a unique point $p^*$.\\
Due to the separate convexity of $L$, its gradient $\nabla L$ will be $0$ at some $p$ in $ \mathcal{S} $ iff $p = p^*$ and $p^* \in \mathcal{S} $. In this "generic" situation, the three equations 
$\partial _u L(p) = \partial _v L(p) = \partial _w L(p) =0$ have a unique solution given by 
\begin{equation}\label{case 1}
u^* = - \frac{bf+2c}{df-4}, \quad v^* = - \frac{2b+cd}{df-4},\quad
w ^*= \frac{a}{2} - \frac{b^2 f+4bc+c^2 d}{4(df-4)}
\end{equation}
Since $df-4 >0$, this solution verifies the constraints \eqref{dom} iff the vector 
$\; \mathcal{A}_N = ( a , b, c, d, f ) \; $ of sufficient statistics verifies the \textit{ genericity} conditions
\begin{equation}\label{genericity}
2b+cd \; < \; 0 \quad \quad 0 \; < \; 2 a(df-4) - b^2 f -4 b c - c^2 d < -4 ( b f + 2 c)
\end{equation}
We shall prove below that for small enough fixed sub-sampling time $T$ , the vector $\mathcal{A}_N $ verifies \eqref{genericity} with a probability tending to 1 as $N \to \infty $. So we consider two cases.\\ 
\textit{Generic Case : } When \eqref{genericity} is true, we estimate the parameter vector $p \in \mathcal{S}$ by $\hat p = p^* = ( u^* , v^* , w^* ) \in \mathcal{S}$ as explicitly given by \eqref{case 1}.\\ \\
\textit{Boundary Case : } When \eqref{genericity} is not true, one can still compute the unique minimizer $p^*$ of $L$ on $\overline{\mathcal{S}} $ by explicit formulas, but then $p^*$ lies on the boundary $\partial \mathcal{S}$, so that we define $\hat p $ as any interior point in $\mathcal{S}$ close enough to $p^*$. As just mentioned, for small enough fixed $T$, the probability of observing a ``boundary case " tends to $0$ when $N \to \infty $. 
\paragraph{Return to the original volatility parameters :}
The change of parameters formula \eqref{uvw2ktg} with $\sigma =T$ naturally suggests to define estimators $( \hat \kappa_N, \hat \theta_N, \hat \gamma_N^2)$ of the original volatility parameters $(\kappa, \theta, \gamma^2)$ by 
\begin{equation}\label{hatuvw2ktg}
\hat \kappa_N = \hat {v}/T,\quad \hat \theta_N = \hat{u}/\hat{v}, \quad 
\hat \gamma_N^2 = 2\hat{w}/T.
\end{equation}
In the generic case where \eqref{genericity} holds, these three estimators are rational functions of the vector $\mathcal{A}_N$ given by, in view of \eqref{case 1} and \eqref{hatuvw2ktg}, 
\begin{equation}\label{hatestim}
\hat \kappa_N = -\frac{2b_N + c_N d_N}{T(d_N f_N - 4)} \;\; ,\quad 
\hat \theta_N = \frac{b_N f_N + 2 c_N}{2b_N + c_N d_N} \;\; , \quad 
\hat \gamma^2_N = \frac{a_N}{T} 
- \frac{b_N^2 f_N + 4b_N c_N + c_N^2 d_N }{2T(d_N f_N - 4)} \;\; 
\end{equation}
The genericity conditions \eqref{genericity} are then clearly equivalent to 
\begin{equation} \label{generic2}
\hat\kappa_N > 0 \quad \quad \text{and} \quad \quad 0 \; < \; \hat\gamma_N^2 \; < \; 2 \hat\kappa_N \hat\theta_N
\end{equation}
which obviously imply $\hat\theta_N > 0$.
\paragraph{Estimators of the drift $\mu$ and the correlation $\rho$ : } \label{driftcorr.estim}
Recall that the processes $U_s = X_{sT}$ and $ V_s = Y_{sT}$ verify the SDEs \eqref{changestock} and \eqref{paruvw} with $\sigma = T \;$ , driven by Brownian motions 
$Z_t$ and $B_t$ with instantaneous correlation $\rho$, and parameters $u,v,w$. To estimate the drift $\mu$ , we discretize SDE \eqref{changestock} to write 
\begin{equation*} 
\Delta Z_n \; \approx \; \frac{\Delta U_n - T\mu U_n }{\sqrt{T} U_n\sqrt{V_n}} \quad \text{where} \quad \Delta Z_n = Z_{n+1} - Z_n \quad \text{and} \quad \Delta U_n = U_{n+1} - U_n \quad 
\end{equation*}
To maximize the log-likelihood, we minimize in $\mu$ the sum $\; \sum_{0 \leq n \leq N-1}(\Delta Z_n )^2 \;$, which provides the estimator $\hat {\mu}_N$ 
\begin{equation}\label{hatmu}
\hat{\mu}_N = \frac{1}{T \sum_{n=0}^{N-1} 1/V_n}\sum_{n=0}^{N-1}\frac{1}{V_n}\frac{\Delta U_n }{U_n}
\end{equation}
We then estimate the Brownian increments $\Delta Z_n $ and $\Delta B_n = B_{n+1} - B_n$ by 
\begin{equation*}
\Delta \hat{Z}_n \approx \dfrac{\Delta U_n - T\hat \mu U_n}{\sqrt{T}\sqrt{V_n}U_n}, \quad \quad \Delta \hat{B}_n \approx \dfrac{\Delta V_n - (\hat u - \hat vV_n)}{\sqrt{2\hat w} \sqrt{V_n}}
\end{equation*}
The empirical correlation $\hat \rho_N$ between $\Delta \hat{Z}_n $ and $\Delta \hat{B}_n$, will then be a natural estimator of $\rho$ 
\section{Distribution invariance for parameter estimators} \label{distrib.invariance}
\subsection{ Canonical Parametrization of Volatility SDE }
\begin{definition}{Canonical volatility parameters : } \label{canonic.param}
Let $Y_t$ be the squared volatility process driven by the Heston SDE \eqref{vol} parametrized by $\kappa, \theta, \gamma^2 $ verifying the constraints \eqref{ParDom}. We will show below that for $N \to \infty$ and fixed sub-sampling time $T$, the asymptotic behaviour of the parameter estimators $\hat\kappa_N, \hat\theta_N, \hat\gamma^2_N$ is controlled by the following two \textit{canonical parameters }
\begin{equation} \label{intrinsic}
\omega = e^{-\kappa T} \quad \text{and} \quad \zeta = \frac{\kappa \theta}{\gamma^2} >1/2
\end{equation}

\end{definition}
In view of paragraphs \ref{scale} and \ref{time}, the canonical parameters $0 < \omega <1 $ and $\zeta>1/2 $ remain unchanged under linear rescaling $ R_t =A Y_t$ or linear time change 
$ Z_s= Y_{s \; \sigma}$.
\begin{definition}{Canonical estimation problem : } \label{defcanonic}
For each $\zeta >1/2$ we define the following \textit{canonical} volatility SDE 
\begin{equation} \label{canonical}
dJ_t = ( \zeta - J_t ) dt + \sqrt{J_t} d W_t 
\end{equation}
where $W_t$ is a \textit{standard} Brownian motion. We call \textit{canonical estimation problem} with canonical parameters $\zeta >1/2$ and $ 0<\omega<1$ the ( virtual ) situation where \\
- the volatility $J_t$ is driven by the canonical SDE \eqref{canonical}, which is of the form \eqref{vol} with ``unknown" parameters $\; \bar\kappa=1, \; \bar\theta = \zeta, \; \bar \gamma= 1$,\\
- the subsampling time is $ \bar {T} = - \log(\omega)$,\\
- the observed data are the $V_n= J_{n \bar{T}} $. \\
\end{definition}
\begin{definition}{Normalized Volatility Parameter Estimators : } \label{defrelacc}
Let $Y_t$ be driven by the volatility SDE \eqref{vol}, parametrized by 
$\; \Theta = ( \kappa, \theta, \gamma^2 ) \; $ verifying \eqref{ParDom}. Given the sub-sampling time $T$ and $N$ successive squared volatility observations $ Y_{nT}$, let 
$\; \hat\Theta_N = ( \hat\kappa_N, \hat\theta_N, \hat\gamma_N^2 ) \;$ be the volatility parameter estimators given by \eqref{hatestim}. \\
Define the vector $D_N$ of \textit{normalized parameter estimators } by
\begin{equation} \label{relaccuracy}
D_N = [ \; \frac{\hat\kappa_N}{\kappa} , \; \frac{\hat\theta_N}{\theta} , \; \frac{\hat\gamma_N^2}{\gamma^2} \; ] 
\end{equation}
\end{definition}
\begin{proposition} \label{invariance}
Hypotheses and notations are those of definition \ref{defrelacc}. Let $D_N$ be the vector of normalized estimators for the three volatility parameters. Then for any $N$ and $T$, the distribution $DIS_N$ of $D_N$ depends only on $N$ and on the canonical parameters $\zeta > 1/2 $ and $ 0 <\omega <1$. Moreover, $DIS_N$ is identical to the distribution of normalized parameter estimators in the canonical estimation problem just defined by \ref{defcanonic}, where the three estimators of $\; \bar\kappa = 1,\; \bar\theta = \zeta,\; \bar \gamma =1 \;$ are computed by formulas \eqref{hatestim} and \eqref{coeff1} with sub-sampling time $\bar{T} = \kappa T$ and volatility data $V_n = J_{n \kappa T}$.
\end{proposition}
\begin{proof}
As seen in sections \ref{scale} and \ref{time}, for any fixed positive $A$ and $\sigma$, the process $\tilde{Y}_s = AY_{s\sigma} $, obtained by linear space rescaling and change of time, verifies a volatility SDE of the same type as \eqref{vol}, but parametrized by 
\begin{equation} \label{23}
\tilde{\kappa} = \sigma \kappa \;\;; \quad \tilde{\theta} = A \theta \;\;; 
\quad \tilde{\gamma}^2 = A \sigma \gamma^2.
\end{equation}
The process $ \tilde{Y}_s $ will now be sub-sampled at time intervals $\tilde{T}= T / \sigma$, to provide the $N$ observations $\tilde{Y}_{n \tilde{T}} = AY_{n T / \sigma}$. We can then compute the five sufficient statistics $\tilde{a}_N, \; \tilde{b}_N, \; \tilde{c}_N, \; \tilde{d_N},\; \tilde{f}_N $ by formulas \eqref{coeff1}. Formula \eqref{hatestim} where $T$ is replaced by $\tilde{T}= T / \sigma \;$, provides the estimators $\; \tilde{\Theta}_N = ( \tilde{\kappa}_N , \tilde{\theta}_N ,\tilde{\gamma}_N^2 ) \;$ of 
$\; \tilde{\Theta} = ( \ \tilde{\kappa}, \tilde{\theta}, \tilde{\gamma^2} )$. The obvious relations
\[
\tilde{a}_N = A a_N , \quad \tilde{b}_N = b_N , \quad \tilde{c}_N = A c_N , \quad \tilde{d}_N = \frac{1}{A} \, d_N , \quad \tilde{f}_N = A f_N 
\]
then imply by \eqref{hatestim} 
\begin{equation} \label{24}
\tilde{\kappa}_N = \sigma \kappa_N; \quad \tilde{\theta}_N = A \theta_N; 
\quad \tilde{\gamma}_N^2 = A \sigma \gamma_N^2.
\end{equation}
The two vectors $D_N$ and $\tilde{D}_N$ of normalized parameter estimators are defined by 
\begin{equation*} \label{240}
D_N = [ \; \frac{\hat\kappa_N}{\kappa} , \; \frac{\hat\theta_N}{\theta} , \; \frac{\hat\gamma_N^2}{\gamma^2} \; ] ; \quad 
\tilde{D}_N = [ \; \frac{\tilde{\kappa}_N}{\tilde{\kappa}} , \; 
\frac{\tilde{\theta}_N}{\tilde{\theta}} , \; 
\frac{\tilde{\gamma}_N^2}{\tilde{\gamma}^2} \; ] \quad 
\end{equation*}
The relations \eqref{23} and \eqref{24} directly imply the \textit{identity } $\tilde{D}_N \equiv D_N$. We have just proved that for any finite $N$ and any $T$ the normalized parameter estimators have the same distributions after arbitrary linear space rescaling and change of time.
In particular let us select $A = \kappa / \gamma^2$ and $\sigma = 1/\kappa$.
We then have 
\[
\tilde{\kappa} = 1 \;\;; \quad \tilde{\theta} = \kappa \theta/\gamma^2 = \zeta > 1/2 \;\;; 
\quad \tilde{\gamma}^2 = 1
\]
so that the process $J_t = \tilde{Y}_t = AY_{t\sigma} = (\kappa / \gamma^2) Y_{t \kappa}$ verifies the \textit{canonical} SDE \eqref{canonical}. This concludes the proof.
\end{proof}
\section{Markov chain of sub-sampled volatilities}
\subsection{Transition density for the volatility process} \label{density}
\textbf{Convention :} We derive below several upper bound inequalities involving various numerical positive constants determined by the volatility parameters and by moment exponents $ p,q > 1 $. To simplify notations, all these positive constants will be denoted by the same letter $C $. \\
The volatility process $Y_t$ is driven by SDE \eqref{vol} with parameters $\kappa, \theta, \gamma^2$ and $T$ is the sub-sampling time. The associated canonical parameters $\zeta >1/2$ and $0 < \omega <1 $ have been defined in \eqref{intrinsic}.\\
As shown in \cite{cox1985theory}, for any $s>0$ and any $t>0$ the conditional density $ g_{t}(y,z) $ of $Y_{s+t}=z$ given $Y_s=y $ is given by
\begin{equation} \label{cox}
g_t(y,z)= \frac{\lambda}{2} \exp( \; - \frac{\lambda}{2} (z + \nu) \;) \; (\frac{z}{\nu})^{r/2} \; I_{r}( \lambda \sqrt{z \nu } )
\end{equation}
where we have set 
\begin{equation} \label{intermediary.notations}
\lambda = \frac{4\kappa}{\gamma^2 (1 - \omega)} \; ; \quad 
r = \frac{2\kappa\theta}{\gamma^2}-1 = 2 \zeta -1 > 0 \;; \quad \nu = e^{-t \kappa}. 
\end{equation}
Here $I_{r}$ is the modified Bessel function of the first kind of order $r$, which has the well known expression $I_r(x)= x^r J(x) $ with $J(x)$ given by the series 
\begin{equation} \label{bessel}
J(x) = \sum_{k=0}^{\infty} \; \frac{x^{2k}}{k! \; \Gamma(k+ 1 + r)}
\end{equation}
where $\Gamma$ is the classical gamma function. \\
We can thus write
\begin{equation} \label{cox.bessel}
g_t(y,z)= \frac{1}{2} \lambda^{r+1} z^r \exp( - \frac{\lambda}{2} (z + y \nu) ) \; J( \lambda \sqrt{z y \nu } ).
\end{equation}
Note also that the transition density $P(R_{s+t} = z \; | \; R_{s} = y) $ of the \textit{rescaled} volatility process $R_{s} = \lambda Y_{s} $ coincides with the density of a non-central chi-square distribution with $ 2(r+1) = 4 \zeta $ degrees of freedom and non-centrality parameter $y \nu$ where $\nu = e^{- t \kappa}$, given by 
\begin{equation} \label{noncentral}
f(z)= \frac{1}{2} \exp( - \frac{z + y \nu}{2} ) (\frac{z}{y \nu})^{r/2} I_{r}(\sqrt{z y \nu}).
\end{equation}
As proved in \cite{cox1985theory, feller1951two}, on the state space $R^+$, the Markov process $Y_t$ has a unique stationary distribution $\Psi$, with density $\psi(z) ) = \lim_{t \to \infty} g_t(y,z)= $ given by 
\begin{equation} \label{psi}
\psi(z) = \frac{1}{2} \lambda^{r+1} z^r \exp( - \frac{\lambda}{2} z ) \; F(0) \; \; \text{for all} \; z > 0.
\end{equation}
After the linear rescaling $Y_s \rightarrow R_s = \lambda Y_s$, the stationary density $\psi(z)$ becomes the density $\; \frac{1}{2^{1+r}\Gamma(1+r)} z^{r} e^{- z /2} \;$ of a standard chi-square distribution having 
$ 2(r+1) = 4 \zeta$ degrees of freedom. 
The following technical proposition quantifies precisely the speed of convergence of $g_t(y,z)$ to $\psi(z)$ as $t \to \infty$. 
\begin{proposition}\label{g.psi.bound}
For fixed $T$, consider the volatility Markov chain $V_n =Y_{nT}$, and let $\Psi$ be its unique stationary probability.
Let $g_t(y,z)$ be the transition density $P( Y_{s+t}= z \; | \; Y_s= y )$ of the continuous time squared volatility process $Y_t$. We then have the upper bound
\begin{equation} \label{cox++}
g_{t}(y,z) \leq C z^r \; [ \; e^{- \lambda z/4} + 1_{ \left\lbrace z \leq 64 y \nu \right\rbrace } \; e^{\frac{3 \lambda}{2} y \nu} \; ] \;\; \text{with} \; \nu = e^{- t \kappa}
\end{equation}
As $t \to \infty$ the transition density $g_t(y,z)$ converges pointwise and at exponential speed to the stationary density $\psi(z)$ of $\Psi$. More precisely, there is a constant $C$ such that for all positive $y, z, t$ , one has with $\nu = e^{-t \kappa}$,
\begin{equation} \label{g-psi}
|g_t(y,z) - \psi(z)| < C \nu y (1+z) z^r \;\left[ \; e^{- \lambda z / 4} + 
1_{ \left\lbrace z \leq 64 \nu y \right\rbrace } \; e^{2 y \lambda \nu }\right] 
\end{equation}

\end{proposition}
\textbf{Proof :} By Stirling formula and since $ \Gamma(k+ 1 + r) \geq k!$, there is a constant $C$ such that for all integers $k \geq 0$, 
\[
\frac{x^{2k}}{k! \; \Gamma(k+ 1 + r)} \leq \frac{x^{2k}}{k! \; k!}\leq
C \frac{(2x)^{2k}}{(2k)!}
\]
The series expansion of $J(x)$ then yields 
\begin{equation} \label{J}
J(x) = \sum_{k=0}^{\infty} \; \frac{x^{2k}}{k! \; \Gamma(k+ 1 + r)} \leq C e^{2 x} \;\; \text{for all} \; x \geq 0
\end{equation}
which by \eqref{cox.bessel} implies the following bound valid for all positive $y, z, t$
\begin{equation} \label{cox+}
g_{t}(y,z) \leq C z^r \exp \; [ \; \lambda ( -z/2 - y \nu /2 + 2 \sqrt{z y \nu} \; ] \;\; \text{with} \; \nu = e^{- t \kappa}
\end{equation}
An elementary argument shows that for all positive $z, y, \nu $, 
\begin{equation} \label{elementary}
-z/2 + 2 \sqrt{ z y \nu } \leq
- \frac{z}{4} 1_{ \left\lbrace z > 64 y \nu \right\rbrace } + 
2 y \nu 1_{ \left\lbrace z \leq 64 y \nu \right\rbrace } 
\end{equation}
Substituting this bound into \eqref{cox+} yields for all positive $z, y, \nu $, the announced upper bound \eqref{cox++} 
\begin{equation} \label{cox+++}
g_{t}(y,z) \leq C z^r \; [ \; e^{- \lambda z/4} 
+ 1_{ \left\lbrace z \leq 64 y \nu \right\rbrace } \; e^{\frac{3}{2}\lambda y \nu} \; ] \;\; \text{where} \; \nu = e^{- t \kappa}
\end{equation}
After the volatility rescaling $R_s= \lambda Y_s$, the densities $g_t(y,z)$ and $\psi(z)$ are replaced by $\frac{1}{\lambda}g_t(y/ \lambda, z/ \lambda)$ and $\frac{1}{\lambda}\psi( z/ \lambda)$. Hence we may and do assume that 
$ \lambda =1 $ to prove \eqref{g-psi}. Equations \eqref{cox.bessel} and \eqref{psi} then become, with $\nu = e^{- t \kappa }$, 
\begin{equation*}
g_t(y,z)= \frac{1}{2} z^r \exp( - \frac{1}{2} (z + y \nu) ) \; J( \sqrt{z y \nu } ) \quad \text{and} \quad \psi(z) = \frac{1}{2} z^r e^{-z/2} J(0)
\end{equation*}
This implies 
\begin{equation} \label{gpsi2}
g_t(y,z)- \psi(z) = 
\frac{1}{2} z^r e^{-z/2} \left[ e^{-y \nu/2} J(z y \nu)-J(0) \right] 
\end{equation}
The series expansion of $J(x)$ yields for all positive $x$
\[
| J(x) - J(0) | = x \sum_{j \geq 0} \frac{x^j}{(j+1)! \; \Gamma(j+2+r)}  \leq x J(x) 
\]
Since $\; | e^{- y \nu/2} -1 | < C y \nu $ for some constant $C$, a standard argument provides then another constant $C$ such that for all positive $\nu, y, z $ 
\begin{equation} \label{eJ-J}
| e^{- y \nu/2} J(z y \nu) - J(0) | < C (1+z) y \nu J(z y \nu) 
\end{equation}
Combining the relations \eqref{gpsi2}, \eqref{eJ-J}, and \eqref{J}, we get another constant $C$ such that for all positive $y, z, t$
\[
|g_t(y,z) - \psi(z)| < C \nu y (1+z) z^r e^{-z/2 + 2 \sqrt{z y \nu}} 
\]
where $ \nu = e^{- t \kappa } $. Applying \eqref{elementary} then yields for all positive $y, z, t$
\[
|g_t(y,z) - \psi(z)| < C \nu y (1+z) z^r \;[\; e^{-z/4}
+ 1_{ \left\lbrace z \leq 64 y \nu \right\rbrace } \; e^{2 y \nu} \;] 
\]
which completes the proof of inequality \eqref{g-psi}.

\subsection{Ergodic theorems for the volatility process}\label{ergodicth}
\paragraph{Notations:} Let $ \mathcal{V}$ be the set of all infinite positive sequences $\{ v_0, v_1, \dots, v_n, \dots \}$ endowed with the sigma-algebra generated by finite products of Borel subsets of $R^+$. \\
For each $y > 0$, and for fixed $T$, the distribution of the infinite random sequence $ V_n = Y_{nT} $ with starting point $V_0 = Y_0 = y$ is a probability $P_y$ on $ \mathcal{V}$, and we let $E_y$ denote expectations with respect to $P_y$. We write $P_y \ a.s.$ for `` $P_y$ almost surely ", and ``almost all $y$" for ``Lebesgue almost all $y$".\\
Under $P_y$, the sequence $V_n = Y_{nT}$ is an homogeneous Markov chain on $R^{+}$, with one-step transition density $ P(V_{n+1}= z \; | \; V_n = y) = g_T(y,z) $ given by formula \eqref{cox} with $t= T$. 
As proved in \cite{feller1951two, cox1985theory}, on the state space $R^+$, the continuous time Markov process $Y_t$ as well as the sub-sampled Markov chain $V_n = Y_{nT}$ both have a unique stationary distribution $\Psi$, with strictly positive and continuous density $\psi$ given by equation \eqref{psi}. When the distribution of $V_0 = Y_0$ is $\Psi$, the Markov chain $V_n$ is \textit{strictly stationary} and irreducible with respect to Lebesgue measure. We call $P_{\Psi}$ its probability distribution on the path space $ \mathcal{V}$.\\
Note that for any set $B \in \mathcal{B}(\mathcal{V})$, one has
\begin{equation} \label{almostsure}
P_{\Psi}(B) =1 \;\; \text{if and only if} \; P_y(B))=1 \;\; \text{for almost all} \; y > 0
\end{equation}
because $P_{\Psi} [ B ] $ is equal to $\; \int_{R^+} P_y [ B ] \psi(y) dy\;$ 
with $\psi(y) $ continuous and positive for $y>0$.\\
\begin{theorem} \label{geometric}
Consider the continuous time squared volatility process $Y_t$ driven by SDE \eqref{vol}. For any fixed $T$, the discrete Markov chain $V_n =Y_{nT}$ is then \textit{geometrically ergodic} in the sense of Ibragimov and Linnik (see definition in \cite{haggstrom2005central} \cite{ibragimov90independent}). More precisely, denote by $\Psi$ the unique stationary probability of $Y_t$ , by $\mathcal{B}(R^+) $ the family of all Borel subsets of $R^+$, and let $\omega = e^{- T \kappa} < 1$. Then for each starting point $y >0$, there is a constant $C(y)$ such that 
\begin{equation} \label{geomergodic} 
\sup_{ B \in \mathcal{B}(R^+)} \; | P_y ( V_n \in B ) - \Psi(B) | < C(y) \omega^n \; \; \text{for all} \; n \geq 0
\end{equation}
The Markov chain $\; \mathcal{V}_n = ( V_n,V_{n+1} ) \;$ on $R^+ \times R^+ $ is also geometrically ergodic with the same geometric rate $ \omega^n $.\\
For the two Markov chains $V_n$ and $V_n,V_{n+1}$, the pointwise ergodic theorem then holds. This means (see \cite{bradley2005basic}),that for any two Borel functions $h(v)$ on $R^+ $ and $k(v,z)$ on $(R^+)^2$ such that 
$\bar{h} = E_{\Psi}( | h(V_n) |$ and $ \bar{k} = E_{\Psi}[\; | k(V_n,V_{n+1}) | \;]$ are finite, then for any starting point $y>0$, we have the $P_y$ almost sure pointwise convergence 
\begin{eqnarray}
\lim_{N \to \infty}& \frac{1}{N} \sum_{n=0}^{N-1} h(V_n) = \bar{h} \quad P_y \; \text{a.s.} \label{h.ergodic} \\
\lim_{N \to \infty} &\frac{1}{N} \sum_{n=0}^{N-1} k(V_n, V_{n+1}) = \bar{k} \quad P_y \; \text{a.s.} \label{k.ergodic}
\end{eqnarray}
Note that the finiteness of $\bar{h}$ and $\bar{k}$ is obviously equivalent to 
\[
\bar{h} = \int_{R^+} \; | h(v) | \psi(v) dv < \infty \;\; \text{and} \;\;
\bar{k} = \int_{(R^+)^2} | k(v,z) | \psi(v) g_T(v,z) dv dz < \infty
\]
\end{theorem}
\begin{proof}
Let $H(z)$ be any Borel function of $z>0$ such that $0 \leq H(z) \leq 1 $. For any Borel subset $B$ of $R^+$, one has 
\[
| E_y ( H(V_n ) - \int_{z>0} \; H(z) \psi(z) dz | \leq \int_{z >0} \; H(z) \; | g_{nT}(y,z) - \psi(z) | dz
\]
Applying \eqref{g-psi} with $t= nT$ , and hence $ \nu = e^{-t \kappa}= \omega^n <1$, we can bound the preceding integral by 
\[
C y \omega^n \;[\; \int_{z >0} \; (1+z) z^r e^{- \lambda z / 4} dz + e^{2 y \lambda \omega^n} \int_{z \leq y}\; (1+z) z^r dz \;] 
\]
The bracketed term in this expression is inferior to $ C (1 +y + y^{1+r}) e^{2y \lambda} $ for some constant $C$, which yields
\begin{equation} \label{11}
| E_y ( H(V_n ) - \int_{v>0} \; H(v) \psi(v) dz | \leq C \omega^n \; e^{3 y \lambda} \; \; \text{for all} \; n \geq 0
\end{equation}
for some constant $C$ which does not depend on the choice of the function $h$ verifying $0 \leq h(z) \leq 1$. \\
Let $B$ be any Borel subset of $R^+$. Selecting $H = 1_B$, inequality \eqref{11}then becomes the relation \eqref{geomergodic} and thus proves the geometric ergodicity of $V_n$. \\
Consider now the Markov chain $\; \mathcal{V}_n = ( V_n,V_{n+1} ) \;$. The stationary measure $\Phi$ of $ \mathcal{V}_n $ on $R^+ \times R^+$ has density $ \phi(v,z) = \psi(v) g_T(v,z)$. \\
Let $W$ be any Borel subsets of $R^+ \times R_+$. Define the function $H$ on $R^+$ by 
\[
H(v) = \int_{z > 0} 1_ W(v,z) g_T(v,z) dz \;\; \text{for all} \; v >0
\]
which clearly verifies $0 \leq H \leq 1$ and 
\[
\Phi(W) = \int_{R^+ \times R^+} \; 1_W(v,z) \psi(v) g_T(v,z) = 
\int_{v>0} \; H(v) \psi(v) dv
\]
We also then have $\; P_y \left[ (V_n,V_{n+1}) \in W \; | \; V_n \right] = H(V_n) \;$, which in turn implies $\; 
P_y ( \mathcal{V}_n \in W ) = E_y( H(V_n) ) \;$. Combining these relations with \eqref{11} shows that 
\begin{equation} \label{12}
| P_y [\; \mathcal{V}_n \in W \;] - \Phi( W ) | \leq C(y) \omega^n \; \; \text{for all} \; n \geq 0
\end{equation}
for all Borel subsets $ W$ of $R^+ \times R^+$. It follows that the chain $\mathcal{V}_n$ is geometrically ergodic. 
By a generic result of Ibragimov-Linnik (see \cite{ibragimov90independent}) the pointwise almost sure ergodic convergence theorem holds for any geometrically ergodic Markov chain, which means that the announced statements \eqref{h.ergodic} and \eqref{k.ergodic} are valid. This concludes the proof of the theorem.
\end{proof}
\subsection{Absolute and conditional moments of the volatility} \label{allmoments}
%
%%%%%%%%%%%%
Our asymptotic study below for the statistics $a_N, b_N, c_N, d_N, f_N$ will require the following uniform moments estimates. 
\begin{lemma} \label{momentlemma}
Fix any $T$ and any pair of exponents $q > 1$ and $1 < p< r+1 = 2 \zeta$. Fix any polynomial $pol(x,z)$ of degree 2 and set $K_t = pol(Y_t, Y_{t + T}) / Y_t $. Then for any starting point $y>0$, and for any $u < \lambda / 4$ the expectations $E_y (Y_t^q)$ , $E_y (e^{u Y_t})$ , and $ E_y ( | K_t | ^p ) $ remain uniformly bounded for all $ t\geq 0$. \\
In the stationary case where the distribution of $Y_0$ is the invariant probability $\Psi$, the 
variables $Y_t^q$ and $ | K_t | ^p $ also have finite expectations under $P_{\Psi}$, which are obviously constant in $t$. 
\end{lemma}
\begin{proof}
To prove these bounds, apply inequality \eqref{cox++} to prove that $E_y (Y_t^q)$ , $E_y (e^{u Y_t})$ , and $E_y (1/Y_t^p)$ are finite and uniformly bounded for $t>0$. Pick $s' >1$ and $s">1 $ such that $p < s' p < 2 \zeta$ and $1/s' + 1/s" =1$. Then $1/Y_t^p$ is in $L_{s'}$ and the polynomial power $| pol(Y_t, Y_{t + T}) |^p$ is in $L_{s"}$. The duality between $L_{s'}$ and $L_{s"}$ then shows that $E_y (K_t^p) $ is finite and uniformly bounded for $t>0$.\\
A similar proof takes care of the strictly stationary case.
\end{proof}
The conditional mean $\; m_y = E( Y_{s+T} \; | \; Y_{s} = y ) \;$ and the conditional variance 
$\; var_y = var( Y_{s+T} \; | \; Y_s =y ) \;$ are easily derived, after rescaling by $\lambda$, from the known mean and variance of the noncentral chi-square density \eqref{noncentral}, to yield 
\begin{eqnarray*}
m_y &=& \int_{R^+} z g_T(y,z) dz = \theta (1-\omega) + \omega y \quad \quad \quad \\
var_y &=& \int_{R^+} (z - m_y)^2 g_T(y,z) dz = \gamma^2 \frac{1-\omega}{\kappa} ( \omega y + (1- \omega) \theta/2 ) \quad
\end{eqnarray*} 
We will also need the conditional means
\begin{eqnarray*}
E(Y_{s+T} - Y_s \; | \; Y_{s} = y )&=& \int_{R^+} (z-y) g_T(y,z) dz = m_y -y \quad \quad \\
E( (Y_{s+T} - Y_s)^2 \; | \; Y_{s} = y ) &=& \int_{R^+} (z - y)^2 g_T(y,z) dz = var_y + (m_y -y)^2 \quad \quad 
\end{eqnarray*}
Replacing $m_y$ and $var_y$ by their values then implies 
\begin{eqnarray} \label{EDV.EDV2}
\int_{R^+} (z-y) g_T(y,z) dz &=& (1-\omega) (\theta - y) \quad \quad \\
\int_{R^+} (z - y)^2 g_T(y,z) dz &=&
y \omega (1-\omega) \frac{\gamma^2}{\kappa} + (1- \omega)^2 ( \frac{\gamma^2 \theta}{2 \kappa } + (\theta - y)^2 ) \quad \quad 
\end{eqnarray}
\section{Asymptotic Bias for $T$ fixed and $N \to \infty $ }
\subsection{Hypotheses and Notations : } \label{notations} We observe a sub-sampled squared volatility process $Y_t$ driven by the Heston SDE \eqref{vol} parametrized by $\kappa, \theta, \gamma^2 $ verifying the constraints \eqref{ParDom}. The sub-sampling time $T>0$ is fixed and known (or pre-imposed by the user). The canonical parameters $\zeta >1/2$ and $0 < \omega <1 $\ are defined by \eqref{intrinsic}.
Given $N$ observations $ V_n = Y_{nT}$ of the squared volatility, we compute the vector $\mathcal{A}_N$ of sufficient statistics $ ( a_N, b_N, c_N, d_N, f_N) $ by \eqref{coeff1}, and the vector of discretized maximum likelihood volatility parameter estimators $ \Theta_N = ( \hat \kappa_N, \hat \theta_N, \hat \gamma^2_N )$ by formulas \eqref{hatestim}.\\
We will use freely the notations just introduced as well as those of \eqref{intermediary.notations}. \\
Recall that $\mathcal{A}_N$ is said to be \textit{ generic} when \eqref{genericity} holds, or equivalently when $ \Theta_N $ verifies the constraints \eqref{ParDom}. 
\begin{definition}{Asymptotic genericity : } We shall say that the sequence $\mathcal{A}_N$ is almost surely asymptotically generic when for each initial $y >0$, the vector $\mathcal{A}_N$ becomes $P_y$-almost surely generic as $N \to \infty$
\end{definition}
\begin{theorem} \label{thmbias}
For fixed sub-sampling time $T$, under the preceding hypotheses and notations, consider $N$ squared volatility data $Y_{nT}$, and the associated vector of 5 sufficient statistics $\mathcal{A}_N$.\\
Then the sequence $\mathcal{A}_N$ is almost surely asymptotically generic if and only if the canonical parameters $\zeta >1/2$ and $\omega= e^{- \kappa T}$ verify one of the two following sets of constraints : 
\begin{eqnarray}
\text{Case (i) } &:& \zeta \geq \frac{3}{4} \; \text{ and $T$ is arbitrary} \label{i}\\
\text{Case (ii) } &:& \frac{1}{2} < \zeta < \frac{3}{4} \; \text{ and } \;\; \omega > \; \frac {\zeta \; (3 - 4 \zeta)} {1 - \zeta} \quad \label{ii}
\end{eqnarray}
Note that in Case (ii) the lower bound on $\omega$ is equivalent to the positive upper bound 
$\; \kappa T < \log \; \frac{1 - \zeta}{\zeta \; (3 - 4 \zeta)}$ , and hence certainly holds for $T$ small enough.\\
Moreover, in both cases (i) and (ii), for every initial $y >0$ and as $N \to \infty$ the volatility parameter estimators $( \hat \kappa_N, \hat \theta_N, \hat \gamma^2_N )$ converge $P_y \; a.s.$ to deterministic limits given by 
\begin{equation} \label{limits}
\kappa_{\infty} \; = \; \frac{1- \omega}{T}; \quad \theta_{\infty} \; = \; \theta; \quad 
\gamma^2_{\infty} = \frac{(1- \omega) \gamma^2}{\kappa T} \; [ \; \omega + (1-\omega) \; \frac{\zeta}{2 \zeta - 1} \; ].
\end{equation}
Hence for fixed $T > 0$ and $N \to \infty$, the estimator $\hat \theta_N$ is asymptotically unbiased, while the estimators $\hat \kappa_N$ and $\hat \gamma_N^2$ both have non zero asymptotic biases. But when $T \to 0$, the asymptotic biases of $\hat \kappa_N$ and $\hat \gamma_N^2$ both tend to zero at the following rates 
\begin{equation*} 
\kappa_{\infty}/ \kappa -1 \; \simeq \; \kappa T/2 \quad \quad \quad 
\gamma^2_{\infty} / \gamma^2 -1 \; \simeq \; \frac{\kappa T}{2} \; \frac{3 - 4 \zeta}{2 \zeta - 1} \quad 
\end{equation*}
\end{theorem}
\begin{corollary} \label{corbias}
Hypotheses and notations are those of theorem \ref{thmbias}. Given $N$ sub-sampled price and squared volatility observations $X_{nT}, Y_{nT} $, define as in section \ref{driftcorr.estim} the estimators $\hat\mu_N$ and $\hat{\rho}_N$ of the drift $\mu$ and the correlation $\rho$. Assume that one of the two asymptotic genericity conditions \eqref{i} or \eqref{ii} is satisfied, Then for almost all initial $y >0$, when $N \to \infty$, the estimators $\hat\mu_N$ and $\hat{\rho}_N $ converge $P_y \; a. s.$ to deterministic limits $\mu_{\infty}$ and $\rho_{\infty}$. Moreover when $T \to 0$, the asymptotic biases $\; ( \mu_{\infty} - \mu ) \;$ and $\; ( \rho_{\infty} - \mu ) \;$ both tend to zero.
\end{corollary}
\begin{proof} 
The proofs of the corollary and of the theorem are quite similar. So we only present the proof of the theorem \ref{thmbias}. \\
For fixed $T$, we have seen above that for the Markov chains $V_n = Y_{nT}$ and $(V_n, V_{n+1})$, the pointwise almost sure convergence ergodic theorems hold $P_y \;$ a.s. for all $y>0$.\\
For $v>0, z>0$, define five functions $\; k_a(v,z) , k_b(v,z) , k_c(v,z) , h_d(z) , h_f(z) \;$ by 
\begin{eqnarray*}
k_a = \frac{(z-v)^2}{v} \; ; \; k_b = -2\frac{(z-v)}{v} \; ; \; k_c = 2 (z-v) \; ; \; h_d = \frac{2}{z} \; ; \; h_f = 2z
\end{eqnarray*}
The five sufficient statistics $a_N, b_N, c_N, d_N, f_N$ are then given by the sums
\begin{eqnarray}
a_N &=& \frac{1}{N} \sum_{n=0}^{N-1} \; k_a(V_n , V_{n+1}); \quad 
b_N = \frac{1}{N} \sum_{n=0}^{N-1} \; k_b(V_n , V_{n+1}); \nonumber \\
c_N &=& \frac{1}{N} \sum_{n=0}^{N-1} \; k_c(V_n , V_{n+1}); \quad 
d_N = \frac{1}{N} \sum_{n=0}^{N-1} \; h_d(V_n ); \quad
f_N = \frac{1}{N} \sum_{n=0}^{N-1} \; h_f(V_n ) \nonumber \\
\end{eqnarray} 
The three functions $k_a, k_b, k_c$ are in $L_1(\Psi)$ and the two functions $h_d, h_f $ are in $L_1(\Phi)$, where $\Psi$ and $\Phi$ are resp. the stationary measures of the Markov chains $V_n$ and $(V_n, V_{n+1})$. Indeed, the explicit form \eqref{psi} of the stationary density $\psi(z)$ yields 
\begin{equation} \label{EzE1/z}
\int_{R^+} h_d(z) \psi(z) dz = 4 \kappa / (2 \kappa \theta - \gamma^2) \quad \text{and} \;\; \int_{R^+} h_f(z) \psi(z) dz = 2 \theta 
\end{equation}
and the explicit integrals of $k_a , k_b, k_c$ with respect to $\phi(v,z) =\psi(v)g_T(v,z) $ are obtained by combining \eqref{EDV.EDV2} and \eqref{EzE1/z} to yield 
\begin{eqnarray*} 
\int_{(R^+)^2} k_a(v,z) \psi(v)g_T(v,z) dv dz &=& (1-\omega) \; \frac{\gamma^2}{\kappa} + (1-\omega)^2 \; \frac{ \gamma^4 }{ \kappa (2\kappa \theta - \gamma^2) }\\
\int_{(R^+)^2} k_b(v,z) \psi(v)g_T(v,z) dv dz &=& -2 (1-\omega)\frac {\gamma^2} { 2 \kappa \theta - \gamma^2 }\\
\int_{(R^+)^2} k_c(v,z) \psi(v)g_T(v,z) dv dz &=& 0
\end{eqnarray*}
So we can now apply the pointwise ergodic convergence theorem to the five statistics $a_N, b_N, c_N, d_N, f_N$ to conclude that for each starting point $y>0$, we have the $P_y $ almost sure limits
\begin{eqnarray*} 
\lim_{N \to \infty} a_N = a_{\infty} &=& (1-\omega) \; \frac{\gamma^2}{\kappa} + (1-\omega)^2 \; \frac{ \gamma^4 }{ \kappa (2\kappa \theta - \gamma^2) }\\
\lim_{N \to \infty} b_N = b_{\infty} &=& -2 (1-\omega)\frac {\gamma^2} { 2 \kappa \theta - \gamma^2 }\\
\end{eqnarray*}
\begin{eqnarray*}
\lim_{N \to \infty} c_N = c_{\infty} = 0 \;; \quad \lim_{N \to \infty} d_N = d_{\infty} = 4 \frac{\kappa}{2 \kappa \theta - \gamma^2} \;; \quad 
\lim_{N \to \infty} f_N = f_{\infty} = 2 \theta
\end{eqnarray*} 
Thus for fixed $T$ and almost all $y> 0 $, we have $ P_y \; a.s.$ the deterministic limit
$$ \lim_{N \to \infty } \mathcal{A}_N = \mathcal{A}_{\infty} = ( \; a_{\infty}, b_{\infty}, c_{\infty}, d_{\infty}, f_{\infty} \; )
$$
For $T$ fixed and $N \to \infty \;$ the vector $\hat \Theta_N$ is of the form 
$\hat \Theta_N = G( \mathcal{A}_N ) $ where $G : R^5 \rightarrow R^3$ is a fixed rational function given by the three explicit formulas \eqref{hatestim}. Hence $\hat \Theta_N$ will also converge $P_y \; a.s. $ to the explicit deterministic limit $\Theta_{\infty} = G( \mathcal{A}_{\infty}) $. Replacing the coordinates of $ \mathcal{A}_{\infty} $ by the explicit values obtained above yields the deterministic expressions of $\kappa_{\infty}, \theta_{\infty}, \gamma^2_{\infty} $ announced in equation \eqref{limits}, and these three limits are positive since $(1-\omega)>0 $ and $2 \kappa \theta - \gamma^2 > 0$. \\
As seen above, for each $N$, the vector of sufficient statistics $\mathcal{A}_N \;$ is generic iff the random variables $\hat \theta_N \;, \hat \gamma_N^2 \;,$ and $\; (2 \hat \kappa_N \hat \theta_N - \hat \gamma_N^2 ) \;$ are all positive. To prove that genericity become $P_y \; a.s.$ valid as $N \to \infty$, one then needs only to verify 
\begin{equation*} 
\theta_{\infty} > 0, \quad \gamma^2_{\infty} > 0, \quad 2 \kappa_{\infty} \; \theta_{\infty} > \gamma^2_{\infty} 
\end{equation*}
The first two inequalities have been proved above; the third one becomes, due to \eqref{limits}, 
\begin{equation} \label{ineq2}
2 \kappa \theta > \gamma^2 \; [ \; \omega + (1-\omega) \; \frac{ \kappa \theta}{2 \kappa \theta - \gamma^2} \; ] 
\end{equation}
which is equivalent to 
\begin{equation} \label{ineq3}
(2\zeta - 1)^2 > (1- \omega) \; (1 - \zeta) \quad \text{where } \;\; 
\zeta = \kappa \theta / \gamma^2 \; > \; 1/2 \quad 
\end{equation}
When $1 \leq \zeta $, inequality \eqref{ineq3} and hence \eqref{ineq2} are true for all $T>0$, so we now only need to study the case $1/2 < \zeta < 1$. \\
Then \eqref{ineq3} holds iff $\omega $ verifies
\begin{equation} \label{ineq5}
\omega > \frac{\zeta \; (3 - 4 \zeta)}{1 - \zeta}
\end{equation} 
When $3/4 \leq \zeta < 1 $, inequality \eqref{ineq5} and hence \eqref{ineq2} will hold for all $\omega \in [ 0 , 1 ] $ and thus for all $T>0$.\\
We have thus proved that for all $T>0$, asymptotic genericity holds $P_y \;$ almost surely when $3/4 < \zeta $.\\
Finally when $1/2 < \zeta < 3/4 $, the constraint \eqref{ineq5} on $\omega = e^{-\kappa T}$ is equivalent to 
\begin{equation*} 
T < \frac{1}{\kappa} \log \; \frac{1 - \zeta}{\zeta \; (3 - 4 \zeta)} )
\end{equation*}
This relation implies the validity of \eqref{generic2}, and hence the $P_y \; a.s.$ asymptotic genericity of $\mathcal{A}_N$ as $N \to \infty$. This concludes the proof of the theorem.
\end{proof}
\subsection{Asymptotically Consistent Estimators }
Since as just proved in theorem \ref{thmbias}, the estimators $ \hat \kappa_N$ and $ \hat \gamma^2_N $ are asymptotically biased, we now define two new asymptotically unbiased estimators $ \mathcal{K}_N $ and $\mathcal{G}_N$ of $ \kappa$ and $\gamma^2$, as explicit non linear functions of $( \hat \kappa_N, \hat \theta_N, \hat \gamma^2_N )$. \\

\begin{theorem} \label{thmnewestim}
For fixed time $T$, consider $N$ squared volatility data $Y_{nT}$, generated by sub-sampling the $Y_t$ volatility process. Hypotheses and notations are those of the preceding section \ref{notations}. Fix any initial volatility $Y_0 = y >0$. Assume that the canonical parameters $\zeta = \kappa \theta / \gamma^2$ and $\omega= e^{- \kappa T}$ verify either one of the two following inequalities
\begin{eqnarray}
\text{Case (i) } &:& \zeta \geq \frac{3}{4} \; \\
\text{Case (ii) } &:& \frac{1}{2} < \zeta < \frac{3}{4} \; \text{ and } \;\; \omega > \; \frac {\zeta \; (3 - 4 \zeta)} {1 - \zeta} 
\end{eqnarray} 
As seen above this guarantees "asymptotic genericity" and the computability of the parameter estimators $( \hat \kappa_N, \hat \theta_N, \hat \gamma^2_N )$ by the generic formulas \eqref{hatestim}. Moreover, due to theorem \ref{thmbias}, the relation $ T \hat \kappa_N < 1 $ becomes $P_y$- a.s. true as $N \to \infty$, and we can define a new estimator $\mathcal{K}_N $ of $ \kappa$ by 
\begin{equation} \label{500K}
\mathcal{K}_N = - \frac{1}{T}\log \; (1 - T \hat \kappa_N)
\end{equation}
Consider the quadratic polynomial $pol_N (Z)$ defined by 
\begin{equation} \label{500pol}
pol_N (Z) = (1-T \hat \kappa_N) Z^2 
+ [ \; \hat \theta_N (T \hat \kappa_N - 2) - \frac{\hat \gamma^2_N}{\hat \kappa_N}] Z 
+ 2 \frac{\hat \gamma^2_N \hat \theta_N }{\hat \kappa_N} 
\end{equation}
Then $P_y$- a.s. , as $N \to \infty$, the two roots $Z_1(N), Z_2(N)$ of $pol_N(Z)= 0 $ become real and verify $0 < Z_1(N) < 2 \hat \theta_N < Z_2(N) $ . \\
We can thus define a new estimator $\mathcal{G}_N $ of $\gamma^2$ by 
\begin{equation} \label{500G}
\mathcal{G}_N = Z_1(N) \mathcal{K}_N
\end{equation}
Then as $N \to \infty$, the three estimators $\mathcal{K}_N, \hat \theta_N, \mathcal{G}_N $ converge $P_y$- a.s. to the true parameter values $ \kappa, \theta, \gamma^2$. In particular as $N \to \infty $, the estimator 
\begin{equation} \label{500Z}
\hat \zeta_N = \mathcal{K}_N \hat \theta_N / \mathcal{G}_N
\end{equation}
converges $P_y$- a.s. to the canonical parameter $\zeta$
\end{theorem}
\begin{proof}
By equation \eqref{limits}, we have $P_y $ \; a.s.
$$ 
\lim_{N \to \infty} \hat \kappa_N \; = \; \frac{1- e^{- \kappa T}}{T }
$$
Solving this relation for $\kappa$ immediately yields $P_y $ \; a.s.
$$ 
\kappa = \lim_{N \to \infty} - \frac{1}{T}\log \; (1 - T \hat \kappa_N) = 
\lim_{N \to \infty} \mathcal{K}_N
$$
Again by equations \eqref{limits}, we have $P_y $ \; a.s.
\begin{equation} \label{501}
\lim_{N \to \infty} \hat \gamma^2_N \; = \; \frac{(1- \omega) \gamma^2}{\kappa T} \; [ \; \omega + (1-\omega) \; \frac{\zeta}{2 \zeta - 1} \; ]
\end{equation}
Introduce the new unknown $z = \gamma^2 / \kappa$. By \eqref{limits} we also have $P_y $ \; a.s. 
\begin{eqnarray*}
\frac{(1- \omega) \gamma^2}{\kappa T} = z \lim_{N \to \infty} \hat \kappa_N, \quad
\omega = \lim_{N \to \infty} ( 1 - T \hat \kappa_N ), \quad
\zeta = \frac {\kappa \theta}{\gamma^2}= \frac{1}{z}\lim_{N \to \infty} \hat \theta_N
\end{eqnarray*}
and $\dfrac{\zeta}{2 \zeta - 1} = \dfrac{1}{2 - z / \hat \theta_N }$.
Substituting these relations into equation \eqref{501}, we see that the expression 
$$
F_N(z) = - \hat \gamma^2_N + 
z \hat \kappa_N \;[ \; (1 - T \hat \kappa_N) + T \hat \kappa_N \frac{1}{2- z / \hat \theta_N} \;] 
$$
converges $P_y $ \; a.s. to 0 as $ N \to \infty $. This is clearly equivalent to stating that the quadratic polynomial $ pol_N(Z) $ defined by equation \eqref{500pol}must verify 
\begin{equation} \label{502}
\lim_{N \to \infty} pol_N(z) = 0 \; ; \; \; P_y a.s.
\end{equation} \
As $N \to \infty $, the coefficients of $pol_N(Z)$ have obvious limits explicitly deduced from equations \eqref{limits}. These limits immediately show that for $N \to \infty$, the value of $pol_N(2 \hat \theta_N)$ becomes negative and the discriminant of $pol_N(Z)$ becomes positive. Hence $P_y$ a.s., for N large enough, the two roots $Z_1(N), Z_2(N)$ of $pol_N(Z)$ are real and verify 
$$ 0 < Z_1(N) < 2 \hat \theta_N < Z_2(N) $$
In view of \eqref{502} it is then easy to conclude that 
\begin{equation} \label{503}
\lim_{N \to \infty} Z_1(N) = z = \gamma^2/ \kappa \; ; \; \; P_y \; a.s.
\end{equation} \
This ends the proof of theorem \ref{thmnewestim}
\end{proof}
\section{Asymptotic distributions of estimation errors}
\subsection{Strong mixing for the volatility process}
For the sufficient statistics $a_N, b_N, d_N, f_N$ computed from subsampled volatility data $V_n = Y_{nT}$ , almost sure convergence as $N\to \infty$ was derived above by applying ergodic theorems to the four sequences 
\[
\; (V_{n+1}- V_n)^2/V_n \;; \quad (V_{n+1}- V_n)/V_n \;; \quad 1/V_n \;; \quad V_n 
\]
To study the asymptotic distributions of these four statistics we need to establish \textit{strong mixing} properties for specific functions of the sub-sampled volatility process $V_n$.
\begin{proposition}\label{strongmix}
Hypotheses and notations are those of Th. \ref{thmbias}. Fix $T$ and consider the sub-sampled squared volatility process $ V_n= Y_{nT} $ with intrinsic parameters $\zeta >1/2$ and $0 < \omega <1 $. 
Fix any starting point $y>0$ and any exponent $q>0$. Under the probability $P_y$ the variables $H_n = V_n^q$ remain uniformly bounded in $L_2$, and for some constant $C$, their covariances verify, 
\begin{equation} \label{decayH} 
| cov(H_{n+j}, H_n) | < C \omega^j \quad \text{for all integers} \;\; n \geq 0 , \; j \geq 0
\end{equation}
Let $K_n$ be random variables of the form $K_n= k(V_n,V_{n+1})$ where $k(y,z)= pol(y,z)/y$ and $pol (y,z)$ is an arbitrary polynomial of degree 2 in $ (y,z) $. \
Under the probability $P_y$, and provided $\zeta>1$, the variables $K_n$ are uniformly bounded in $L_2$, and for some constant $C $ their covariances verify
\begin{equation} \label{decayK} 
| cov(K_{n+j}, K_n) | < C \omega^j \quad \text{for all integers} \;\; n \geq 0 , \; j \geq 0
\end{equation}
In the strictly stationary case where $V_0$ has the unique invariant distribution $\Psi$ with density $\psi$ (see \eqref{psi}), the two results \eqref{decayH}and \eqref{decayK} remain valid under the probability $P_{\Psi}$. 
\end{proposition}
\begin{proof}
The validity of results \eqref{decayH}and \eqref{decayK} is obviously invariant when we replace $Y_t$ by $AY_t$. By equation \eqref{barchange} this rescaling leaves $\zeta$ and $\omega$ unchanged but replaces $\lambda = \frac{4\kappa}{\gamma^2 (1 - \omega)}$ by $\lambda/ A$. So before proving the proposition we implement this rescaling with $A= \lambda$, and still keep the notation $V_n= Y_{nT}$ for the rescaled process. From now on, we thus have $\lambda=1$.\\
We now give a detailed proof only for the more delicate case of the variables $K_n$ under the strictly stationary probability $P_{\Psi}$, and under the ( unavoidable ) hypothesis $\zeta >1$. All the other strong mixing statements in prop. \ref{strongmix} can be derived by quite similar proofs, omitted here for brevity.\\
Due to lemma \ref{momentlemma}, $K_n = pol(V_n, V_{n+1}) / V_n$ must be in $L_2$ since $2 \zeta > 2$. Define the conditional expectation 
\[
h(y) = E(K_n | Vn=y) = \int_{z>0} \; dz \; g_T(y,z) k(y,z) 
\] 
Since $k(y,z) =pol(y,z)/y$, where $pol$ is a polynomial of degree 2, equation \eqref{EDV.EDV2} shows directly that $| h(y) | < C (1 + y + 1/y) $ for some constant C 
so that $\; M = E(K_n) = \int_{y >0} \; dy \; \psi(y) h(y) \;$
Let $\mathcal{F}_n$ be the sigma-algebra generated by $V_0, \ldots, V_n$. \\
For $j >1$ the relation $\; E ( K_{n+j} \; | \; \mathcal{F}_{n+j} ) = h(V_{n+j}) \;$ entails 
\[
E ( K_{n+j} \; | \; \mathcal{F}_{n+1} ) = E ( h(V_{n+j}\; | \; \mathcal{F}_{n+1} ) = G_t(V_{n+1})
\]
where $t= (j-1)T$ and $\; G(y) = \int_{z >0} \; dz h(z) g_{(j-1)T}(y,z) $. The function 
\[
H_t(y) = G_t(y) - M = \int_{z >0} \; dz h(z) [\; g_t(y,z) - \psi(z) \;]
\]
verifies then $\; E ( K_{n+j} - M \; | \; \mathcal{F}_{n+1} ) = H_t(V_{n+1} ) \;$, which yields 
\begin{equation} \label{18}
cov(K_{n+j}, K_n)= E [\; (K_n - M)(K_{n+j}- M) \;] \; = E ( K_n H_t(V_{n+1}) ) 
\end{equation}
In view of lemma \ref{momentlemma} and of the bound $h(z) | < 1+z +1/z $, the definition of $H_t$ implies, with the notations $\nu = y e^{- t \kappa}$ and $t= (j-1) T$, 
\begin{equation} \label{19}
| H_t(y) | \leq C \nu \; [ \; \int_{z >64 \nu} \; (1 +z +1/z) (1+z) z^r e^{-z/4}+ \int_{z \leq 64 \nu} (1 +z +1/z) (1+z) z^r e^{2 \nu} \;] 
\end{equation} 
Since $\; r-1= 2 \zeta -2 >0 \;$ , the first integral in \eqref{19} is bounded by a constant. The 2nd integral is bounded by $\;C(\nu^{r-1} + \nu^r +\nu^{r+2}e^{2 \nu} ) \;$ and a fortiori by $C e^{3 \nu}$. Hence \eqref{19}finally implies 
\begin{equation} \label{20}
| H_t(y) | \leq C \nu e^{3 \nu} = C y e^{- t \kappa} \exp( 3 y e^{- t \kappa} ) 
\end{equation}
Under the probability $P_{\Psi}$, one has $\; E( H_t(V_{n+1})^2 ) = \int_{y>0} \; \psi(y) H_t(y) ^2 dy \;$. This implies in view of \eqref{20} and \eqref{psi},

\[
E( H_t(V_{n+1})^2 ) \leq C e^{- 2 t \kappa} \int_{y>0} y^{r+2}exp(-y/2)\exp( 6y e^{- t \kappa} )
\]
This last integral remains clearly bounded by a finite constant for all $t = (j-1)T$ such that $ e^{- t \kappa} < 1/24 $. Hence for all $j > 1 + \frac{1}{T} \log(24)$, the norm $ || H_t(V_{n+1}) ||_{L_2}$ is inferior to $C e^{- t \kappa} = C \omega^{j-1} $. Since $|| K_n ||_{L_2}$ is a finite constant, the expressions of covariances obtained in \eqref{18} entail the exponential decay inequality 
\[
| | cov(K_{n+j}, K_n)| | \leq C \omega^{j}\;\; \text{for all positive integers $n$ and $j$}
\]
\end{proof}
\subsection{Dichotomy between Gaussian and Stable asymptotics } Due to prop.\ref{invariance} for the Heston volatility SDE \eqref{vol}, the distribution of 
$D_N = \; [\; \hat\kappa_N/\kappa , \hat\theta_N/\hat\theta , \hat\gamma_N^2/\hat \gamma^2 \;] \;$ depends only on $N$ and on the canonical parameters $\; \zeta > 1/2 \;$ and $0 < \omega <1 \;$. By Th.\ref{thmbias}, asymptotic genericity is always true when $\; \frac{3}{4} \leq \zeta \;$, but when $\; \frac{1}{2} < \zeta < \frac{3}{4} \;$, asymptotic genericity holds iff $T$ is small enough. In these two situations, we have also seen that $D_N$ converges almost surely to a deterministic limit $D_{\infty}$. We now seek to adequately rescale the distribution of $D_N - D_{\infty}$ to ensure convergence in distribution as $N \to \infty$.\\
As seen in section \ref{ergodicth} the vector of four statistics $a_N, b_N, d_N, f_N $ is of the form $\frac{1}{N} \sum_{n=0}^{N-1} k_n$ with $k_n = k(V_n, V_{n+1})$, where the vector valued function $k(y,z)$ is $\; (z-y)^2/y , \; (z-y)/y ,\; 1/y , \; 2 y \;$.
Due to the strong mixing properties proved in the last paragraph, we may expect that as $N \to \infty$ the vector of statistics $a_N, b_N, d_N, f_N $ should behave roughly like the average of i.i.d. random vectors having the same finite moments as the $k_n$. By lemma \ref{momentlemma}, we know that the corresponding $k_n$ have finite moments of order $p>1 $ if and only if $1< p < 2 \zeta$.\\
We shall also prove below that for any $s<1$, the statistic $c_N$ verifies almost surely $ \lim_{N \to \infty} \; N^s c_N = 0$. \\
These remarks lead us to expect a radical dichotomy between the two cases $\zeta >1$ and $1/2 < \zeta \leq 1$. For $\zeta >1$, we will prove below that $N^{1/2} (D_N -D_{\infty})$ is asymptotically Gaussian. But for $1/2 < \zeta \leq 1$ we conjecture below that there is an explicit exponent $q = q(\zeta) )<1/2 $ such that the rescaled variables $N^{q} (D_N -D_{\infty})$ converge in distribution to limit distributions with heavy tails similar to the tails of stable distributions.

\subsection{The asymptotically Gaussian case : $\zeta >1$ }
\begin{theorem} \label{gaussian.asymptotics} Hypotheses and notations are those of Th. \ref{thmbias}. Assume that the canonical parameter verifies $\zeta >1 $. Let $A_N = ( a_N, b_N, c_N, d_N, f_N ) $ be the vector of sufficient statistics given by \eqref{coeff1}, let $\Theta_N = (\hat \kappa_N, \hat \theta_N, \hat \gamma^2_N)$ be the vector of volatility parameters estimators computed from $A_N$ by formulas \eqref{hatestim}, and let $U_N = (\mathcal{K}_N, \hat \theta_N, \mathcal{G}_N)$ be the vector of asymptotically unbiased estimators constructed from $\Theta_N$ in theorem \ref{thmnewestim}. Fix a starting point $Y_0 = y > 0$. Then as $N \to \infty$, the three vectors $\; A_N , \Theta_N , U_N \;$ converge $P_y$-a.s. to deterministic limits $A_{\infty}$, $\Theta_{\infty}$, $U_{\infty}$, computed in Th. \ref{thmbias} and Th. \ref{thmnewestim}. Recall that $U_{\infty} $ is the true parameter vector $ ( \; \kappa , \theta , \gamma^2 \; )$.\\
Then under the probability $P_y$, as $N \to \infty$ the random vectors $N^{1/2} (A_N - A_{\infty}) $, $N^{1/2} (\Theta_N - \Theta_{\infty}) $ and $N^{1/2} (U_N - U_{\infty}) $ are asymptotically Gaussian with zero mean. 
\end{theorem}
\begin{proof}
For fixed $T$, by th. \ref{g.psi.bound}, the Markov chains $V_n= Y_{nT}$ and $(V_n,V_{n+1})$ are both geometrically ergodic. Consider arbitrary Borel functions $ h(y)$ and $k(y,z)$ of $(y, z) \in R^2$ , and taking their values in any euclidean vector space. Assume that the coordinates of the random vectors $h_n= h(V_n)$ and $k_n= k(V_n,V_{n+1})$ belong to $L_p$ for some $p >2$. Define the random vectors $ H_N$ and $K_N $ given by 
\[ 
H_N = \frac{1}{N}\sum_{n \geq 0} \; h_n \quad \text{and} \quad K_N = \frac{1}{N}\sum_{n \geq 0} \; k_n
\] 
For functions of geometrically ergodic Markov chains, classical results of Ibragimov- Linnik (see \cite{haggstrom2005central}, \cite{ibragimov90independent}) show that $ H_N$ and $K_N $ converge almost surely to deterministic limits $H_{\infty}$ and $K_{\infty}$, and that $N^{1/2} (H_N - H_{\infty})$ as well as $N^{1/2} (K_N - K_{\infty})$ converge in distribution to centered Gaussian distributions.\\
Ibragimov-Linnik's results were initially stated for real valued functions $h(y)$ and $k(y,z)$, but they are quite easily extended to the situation where $h(y)$ and $k(y,z)$ are vector valued.\\
Let us apply this result to the function 
$$
k(y,z) = [ (z-y)^2/y , (z-y)/y , 1/y , 2 y ]
$$
and the Markov chain $(V_n,V_{n+1})$. Due to formulas \eqref{coeff1}, the average $K_N $ of the N vectors $ k_n= k(V_n,V_{n+1}) $ for $n=1$ to $N$ is then the vector of 4 statistics 
$( a_N, b_N, d_N, f_N )$. Lemma \ref{momentlemma} shows that the four coordinates of $k_n$\ are in $L_p$ for all $p < 2\zeta$ and hence for some $p >2$ since $ \zeta >1 $. \\
Applying the just quoted Ibragimov-Linnik results to $K_N = ( a_N, b_N, d_N, f_N)$ and its limit $K_{\infty} = ( a_{\infty}, b_{\infty}, d_{\infty}, f_{\infty} )$, we conclude that under $P_y$, as $N \to \infty$,the random vectors $N^{1/2} (K_N - K_{\infty}) $ converge in distribution to a four-dimensional centered Gaussian.\\
We now study the last sufficient statistic $c_N = \frac{2}{N} (V_N - V_0) $. Fix $u$ such that $ 0 < u < \lambda/4$. By lemma \ref{momentlemma}, for each fixed $y >0$, one can find a constant $C(y)$ such that 
\[
E_y ( e^{u V_n } ) = \int_{z>0} \; g_{nT} (y,z) e^{u z } dz \leq C(y) \;\; \text{for all} \; n \geq 0 
\] 
This implies $P_y (V_n > 2 \log ( n ) / u \leq C(y) / n^2$. Hence by Borel-Cantelli lemma, there is an almost surely finite random integer $N_0$ such that $V_n < 2 \log(n) / u$ for all $n>N_0$. It follows that for any $s <1$ the sequence $N^s c_N = N^s (V_N - V_0) / N)$ converges $P_y$ a.s. to $0$ as $N \to \infty$.\\
Combining the fast convergence to $0$ of $N^{1/2} c_N$ with the asymptotic normality of $K_N$, we conclude that the random vectors $A_N= ( a_N, b_N, c_N, d_N, f_N ) $ are asymptotivally Gaussian, and that the 5-dimensional random vectors $N^{1/2} (A_N - A_{\infty}) $ converge in distribution to a centered Gaussian distribution concentrated on a 
4-dimensional subspace of $R^5$. \\
By construction , $\Theta_N = (\hat \kappa_N, \hat \theta_N, \hat \gamma^2_N)$ is an explicit smooth function of $A_N$, and for $N$ large enough $U_N$ is an explicit smooth function of $\Theta_N$. Since as is well known, smooth functions preserve asymptotic normality, it follows that $N^{1/2} (\Theta_N - \Theta_{\infty}) $ and $N^{1/2} (U_N - U_{\infty}) $ both converge in distribution to centered Gaussian distributions.
\end{proof}
\subsection{The asymptotically stable case : $1/2 <\zeta \leq 1$ }
\textbf{Conjecture :} \;\; Under the hypotheses of Th. \ref{thmbias}, assume now that the canonical parameter $\zeta$ verifies $1/2 < \zeta < 1$. Then we conjecture that there is an explicit positive exponent $q = q(\zeta) < 1/2$ such that, as $N \to \infty$, the random vectors 
$ N^q (A_N - A_{\infty}) $ and $ N^q (\Theta_N - \Theta_{\infty}) $ and 
$ N^q (U_N - U_{\infty}) $ converge in distribution to limit laws having the same heavy tails as the classical "stable distributions".\\
We have obtained preliminary validation of this conjecture by intensive simulations conducted as follows, but these simulation results will be presented elsewhere. 
\subsection{Computation of Empirical Accuracies} \label{computation}
Let $DIS_N$ be the probability distribution of relative accuracies for the volatility parameter estimators $ \hat{\kappa}_N, \hat{\theta}_N, \hat{\gamma}_N^2 $. For fixed $T,N$, and known parameters $\kappa,\theta,\gamma^2$, a natural goal is to construct good numerical approximations of the distribution $DIS_N$.\\
In view of prop. \ref{invariance} one can proceed by intensive simulations of the process $J_t$ driven by the canonical SDE \eqref{canonical} parametrized by $\; \tilde{\kappa} = 1 , \tilde{\theta} = \zeta ,\tilde{\gamma} =1 \;$. By Euler discretization of SDE \eqref{canonical} with small time step $\delta$, we simulate $5,000$ long trajectories of $J_t$. Each trajectory is sub-sampled at time intervals $ \tilde{T} = \kappa T >> \delta $ to generate $N$ virtual observations $J_{n \kappa T}$, and thus yields one virtual value for the estimators $\; \tilde{\kappa}_N , \tilde{\theta}_N ,\tilde{\gamma}_N^2 \;$ of the canonical parameters $(1, \zeta, 1) $. This generates $5000$ values for the canonical relative accuracies, which by prop. \ref{invariance} provide an approximate empirical histogram of $DIS_N$. \\
We present below numerical results derived from such empirical histograms.\\ 
To simulate the canonical SDE \eqref{canonical}, fix a global observation time $S = N \kappa T$ and a small time step $\delta << \kappa T$, pick any starting point $y_0 >0$, and implement the recursive Euler discretization 
\[
y_{k+1} - y_k = \delta (\zeta - y_k) + \sqrt{y_k} \sqrt{\delta} G_k 
\]
where the $G_k$ are independent standard Gaussian random variables. This is done as long as $y_n >0$ and $n \leq S / \delta$. If $y_n$ becomes negative for some $n< S/\delta$ the trajectory is dismissed.\\
In our numerical explorations, $S / \delta$ ranged from 2,500 to 500,000, and we selected 
$\delta = \kappa T / m $ for integers m between 10 and 20.\\
To validate theoretically this simulation scheme, note that for $S$ fixed, the $L_2$ norms of $(Y_{n \delta} - y_n )$ tend to 0 as $\delta \to 0$, uniformly for all $n \delta \leq S$ (see \cite{higham2005convergence}). For alternate simulation methods, see for instance \cite{broadie2006exact, andersen2007efficient}.\\
For the estimators $ \hat{\mu}_N, \hat{\rho}_N$ of the price SDE parameters $\mu,\rho$, the estimators accuracies are numerically evaluated by a similar Euler discretization and simulation of the joint Heston SDEs.
%%%
\section{Fitting Heston SDEs to market data}
\subsection{Fitting Heston SDEs to S\&P 500 daily data}
We start with $ N= 252$ daily closing values $SPX_n, VIX_n$ of the S\&P 500 and VIX indices, recorded from day 01/03/2006 to 12/29/2006. The annualized squared volatility of SPX is as usual approximated by $VIX^2$. The CBOE database (\cite{CBOE},\cite{carr2006tale} ) computes VIX daily by 
$ VIX_n^2 = A \; \sigma_n $ where $A= 365$ and $\sigma_n$ evaluates the variance of $\; (SPX_n - SPX_{n-1} ) / SPX_{n-1}) \;$ by filtering over 30 days. \\
To model these $N$ data by sub-sampled joint SDEs, we fix the sub-sampling time at the commonly used value $T=1/252$, so that the global observation time is $S= NT = 1$. We seek a Heston process $(X_t,Y_t)$ such that $\; SPX_n = X_{nT} \; $ and $\; VIX_n^2 = A \; Y_{nT} = R_{nT} \;$ where $\; R_t = A \; Y_t \;$ is the annualized squared volatility of $X_t$. As seen in section \ref{scale}, we can model the SDEs driving $( X_t, R_t )$ by 
\begin{equation} \label{spvix}
dX_t /X_t = \mu dt + \frac{1}{\sqrt{A}} \sqrt{R_t} d Z_t \quad \quad \text{and} \quad \quad
dR_t = \kappa( \theta - R_t) dt + \gamma \sqrt{R_t} d B_t \quad 
\end{equation} 
For these $N$ joint data, the sufficient statistics $a_N,b_N,c_N,d_N,f_N$ given by \eqref{coeff1} do verify the genericity condition \eqref{genericity}. The parameters of \eqref{spvix} are then estimated by formula \eqref{hatestim} to yield 
\begin{equation}\label{SPparameters}
\hat \kappa_N = 16.6 ; \quad \hat \theta_N = 0.017 ; \quad \hat \gamma_N = 0.28 ;
\quad \hat \rho_N = - 0.54,\quad \hat \mu_N = 0.126
\end{equation}

The canonical parameters $\omega$ and $\zeta$ thus have the estimated values $\hat \omega_N = 0.936$ and $\hat \zeta_N = 3.599$ so that for this estimated Heston model, the parameter estimators should be asymptotically normal for $N$ large enough.\\
The negative correlation $\hat \rho_N = - 0.54$ between $Z_t$ and $B_t$ indicates a `skew' or `leverage' effect.\\
Let RMS stand for "Root Mean Squared". We compute the RMS estimation errors on $\; \kappa, \theta, \gamma, \rho\; $ as outlined in section \ref{computation}, by simulating 5000 joint SDEs trajectories of $(X_t,R_t)$ for $0 \leq t \leq 1$ and extracting $N=252$ sub-sampled points. This yields 
\[
RMS_{\kappa} = 5.7 \;;\; RMS_{\theta}= 0.002 \;;\; RMS_{\gamma} = 0.01 \;;\; RMS_{\rho} = 0.06 
\]
The \textit{relative} RMS estimation errors for $\;\kappa, \theta, \gamma, \rho\;$ are equal to $34 \%, 12\%, 4\%, 11 \% $, which are rather high due to the quite small number $N=252$ of data. The relative RMS error on the asset price drift $\mu$ is very high, and shows that for $N=252$ both the estimate of $\mu$ and of RMS{$\mu$}  cannot be used, but fortunately $\mu$ plays no part in the well known option pricing PDEs.\\
\subsection{Fitting Heston SDEs to Intra-Day Data}
We consider $N=510$ Credit Agricole stock price data $X_{nT}$ recorded at time intervals $T = 1$ minute on trading day 05/06/2010. The global observation time is $S=NT=510$.\\
The instantaneous squared volatility $Y_{nT}$ of $ X_{nT}$ is estimated (see \cite{garman1980estimation}) by the Garman-Klass formula 
$\;\;\; Y_{nT} \simeq 0.5 \; ( H_{nT} - L_{nT} )^2 - 0.386 \; (Q_{nT})^2 \;\;\;$, where $\; (H_{nT} , L_{nT} , Q_{nT}) \;$ are the highest, lowest, and last stock prices in the current one minute time slice. \\
Rescaling by the annualization factor $A= 60 \times 24 \times 365 = 525600 $ replaces $Y_{nT}$ by the annualized squared volatility $R_{nT} =A Y_{nT}$. As above, we seek to fit to these $N= 510$ annualized data $X_{nT}, R_{nT}$ a Heston process $X_t, R_t$ driven by joint SDEs of the form \eqref{spvix}.\\
These data do satisfy the genericity condition \eqref{genericity}, and formulas \eqref{coeff1} , \eqref{hatestim} directly yield the following parameter estimates for SDEs \eqref{spvix} 
\begin{equation*}
\hat \kappa_N= 0.48,\quad \hat \theta_N= 3.15, \quad 
\hat \gamma_N = 0.80, \quad \hat \rho_N= - .04 , \quad \hat \mu_N = 3\;10^{-5}
\end{equation*}
The canonical parameters $\omega$ and $\zeta$ are then estimated by $\hat \omega_N = 0.619$ and $\hat \zeta_N = 2.362$ so that for this Heston model, the parameter estimators should be asymptotically normal for $N$ large enough.\\
As above, we evaluate the associated RMS estimation errors by 5000 simulations of the Heston SDEs just fitted to these 510 annualized data. This provides the estimates 
\[
RMS_{\kappa} = 0.10 \;;\; RMS_{\theta}= 0.11 \;;\; RMS_{\gamma} = 0.14 \;;\; RMS_{\rho} = 0.04 \;;\; RMS_{\mu} = 8 \;10^{-5}
\]
The relative root mean squared errors on $\hat\kappa_N, \hat\theta_N, \hat\gamma_N$ are $21\%, 3.5\% ,18\%$. As for $\hat \rho_N$ and $\hat \mu_N$ the results indicate that one should pre-impose $\rho =0$ and $\mu =0 $ to model these data. 
\section{Numerical results on estimators accuracy }
\subsection{Simulations of four canonical Heston volatility SDEs}
By linear space and time rescaling as in section \ref{distrib.invariance} the S\&P 500 volatility SDE parametrized by \eqref{SPparameters} becomes the canonical SDE \eqref{canonical} with canonical parameters $\omega=0.936$, $\zeta= 3.599$, and subsampling time $T = 16.6 / 252 = 0.0659$.\\
In view of this concrete example, we have simulated long volatility trajectories $Y_t$ for four canonical Heston volatility SDEs corresponding to a fixed $\omega = 0.936$ and a fixed sub-sampling time $T = 0.0659$, successively combined with 4 values $ 1.1, 1.5 , 2.5 , 3.5$ for $\zeta$. Hence these four SDEs are parametrized by $\kappa =\gamma=1$ and $\theta= \zeta = 1.1 , 1.5 , 2.5 , 3.5$. \\
Note that for all these examples we have $\zeta >1$ and hence all our parameter estimators are asymptotically normal.\\
For each one of these four canonical SDEs, we have simulated 1100 trajectories, each one of which involved $200,000$ discrete steps of duration $T/20$. Each trajectory was subsampled at times $nT$ to extract 10,000 volatility data $Y_{nT}$. \\
For $N= 250, 500, 750, \ldots, 10,000$, these simulated data were then used to compute the sequences of nearly MLE estimators $\Theta_N = ( \hat\kappa_N, \hat\theta_N, \hat\gamma_N )$ and of consistent estimators $U_N = ( \mathcal{K}_N, \hat\theta_N, \mathcal{G}_N )$.
\subsection{Relative root mean squared errors of estimation }
Let $\eta$ be any one of the three parameters $\kappa, \theta,\gamma^2$ and let $\eta_N$ be any estimator of $\eta$, based on N subsampled observations $Y_0, Y_T, \ldots, Y_{(N-1)T}$.
We then characterize the relative accuracy of the estimator $\eta_N$ by the \textit{relative} root mean squared error of estimation, classically defined by
$$
\sigma(\eta_N) = || \; \eta_N - \eta \; ||_2 \, / \eta
$$
Here we naturally estimate these relative root mean squared errors $\sigma(\eta_N) $ by empirical averages over 1100 simulated trajectories, and their asymptotic behaviour as $N \to \infty$ depend only on the pair $\omega, \zeta$ of canonical parameters, due to the crucial scale invariance results of prop. \ref{invariance}.\\
For $T = 0.0659$ and $\omega =0.936 $, our simulations evaluate the accuracies of our five parameter estimators 
$ \; \hat\kappa_N , \mathcal{K}_N , \hat\theta_N , \hat\gamma_N , \mathcal{G}_N \; $ for $250 \leq N \leq 10000$ and for the four values $\zeta = 1.1 ,1.5 , 2.5 , 3.5$. A summary of our numerical results is provided below by the two Tables \ref{table.zeta: 1.5} and \ref{table.zeta: 3.5}.\\
\subsection{Rates of decrease for root mean squared errors}
For each $\zeta$, and for all $N > 1000$, the relative errors $\; \sigma( \mathcal{K}_N )$ , $\sigma(\hat\theta_N)$ , $\sigma( \mathcal{G}_N )$ of our three consistent estimators are very well approximated by $C_1 / N^{1/2}$ , $C_2 / N^{1/2}$, $C_3 / N^{1/2}$ , where the constants $C_1, C_2, C_3$ depend on $\zeta$ but remain quite moderate as shown below in Table \ref{table.sqrtN}. We note also that for each estimator and each $N > 1000$, the relative errors of estimation decrease when $\zeta > 1$ increases, and they practically stabilize as soon as $\zeta > 3$.\\
\begin{table}[h]
\centering
%\begin{tabular}{|c|c|c|c|c|} 
\begin{tabular*}{\columnwidth}{@{\extracolsep{\fill}}|c|c|c|c|c|}
\hline
$\zeta$ & 1.1 & 1.5 & 2.5 &3.5 \\
\hline
$N^{1/2} \, \sigma(\mathcal{K}_N) \; \simeq $ & 6.5 & 6 & 5.7 & 5.7 \\
\hline
$N^{1/2} \, \sigma(\theta_N) \; \simeq $ & 3.7 & 3.2 & 2.5 & 2.1 \\
\hline
$N^{1/2} \, \sigma(\mathcal{G}_N) \; \simeq $ & 1.65 & 1.55 & 1.5 & 1.5 \\
\hline
\end{tabular*}
\caption{We display the values of the constantes $C$ giving good approximations in $C / N^{1/2}$ for the relative root mean square errors of our three consistent parameters estimators $ \mathcal{K}_N , \hat\theta_N , \mathcal{G}_N $. These approximations are quite accurate for $N >1000$ volatility observations, with sub-sampling time $T= 0.0659 $. The canonical parameters take the values $\omega= 0.936$, and $\; \zeta= 1.1, 1.5, 2.5, 3.5 \;$.}
\label{table.sqrtN}
\end{table}

\subsection{Detailed analysis of parameter estimators accuracies}
Our detailed numerical results, which are summarized below in Tables \ref{table.zeta: 1.5} and \ref{table.zeta: 3.5}, yield the following qualitative conclusions.
\paragraph{Parameter $\kappa$ :} The relative errors of estimation for the consistent estimator $\mathcal{K}_N$ remain very slightly larger than for the biased estimator $\hat \kappa_N$ for $N \leq 10,000$. Indeed here the asymptotic relative bias of $\hat \kappa_N$ is quite small, of the order of 2.5\%, and the theoretical advantage of the consistent estimator $\mathcal{K}_N$ over $\hat \kappa_N$ only emerges for unrealistic numbers of observations $N > 40,000$ . However when $N$ increases, approximate normality becomes valid quite sooner for $\mathcal{K}_N$ than for $\hat \kappa_N$.
\paragraph{Parameter $\theta$ :} The asymptotically unbiased estimator $\hat \theta_N$ is quite accurate even for moderate values of N, and approximate normality becomes valid as soon as $N \geq 500$.
\paragraph{Parameter $\gamma^2$ :} The relative errors of estimation for the consistent estimator $\mathcal{G}^2_N$ are clearly smaller than for the biased estimator $\hat \gamma^2_N$ for $N \leq 10,000$. Indeed here the asymptotic relative bias of $\hat \gamma^2_N$ is rather large, of the order of 5\% to 5.5\%, and the theoretical advantage of the consistent estimator $\mathcal{G}^2_N$ over $\hat \gamma^2_N$ is manifest for all values of $N$. Moreover for $\mathcal{G}^2_N$, approximate normality becomes valid as soon as $N \geq 500$.

\newpage

\subsection{Numerical results for estimators accuracies }
The relative root mean squared errors of our five estimators are displayed in the following two tables, where the sub-sampling time T and the canonical parameter $\omega$ are kept fixed at $T= 0.0659$ and $\omega= 0.936$. These error sizes are hence given as percentages of the true parameter value, and the number of volatility observations is restricted to the five levels $N = 500, 1000, 2500, 5000, 10000$. \\

\begin{table} [h]
\centering
%\begin{tabular}{|c|c|c|c|c|c|}
\begin{tabular*}{\columnwidth}{@{\extracolsep{\fill}}|c|c|c|c|c|c|}
\hline
N & 500 & 1000 & 2,500 &5,000 &10,000 \\
\hline
$\sigma(\hat \kappa_N)$& 28 \% & 18 \% & 11 \% & 8 \% & 6 \%\\
\hline
$\sigma(\mathcal{K}_N) $& 32 \% & 20 \% & 12 \% & 8 \% & 6 \% \\
\hline
$\sigma(\theta_N)$ & 15 \% & 10 \% & 6 \% & 4 \% & 3 \%\\
\hline
$\sigma(\hat \gamma^2_N)$ & 8 \% & 6 \% & 5 \% & 5 \% & 5 \% \\
\hline
$\sigma(\mathcal{G}^2_N)$ & 7 \% & 5 \% & 3 \% & 2 \% & 1 \%\\
\hline
\end{tabular*}
\caption{Relative root mean squared errors of volatility parameters estimators for $\zeta= 1.5$. When $\zeta$ decreases to $\zeta =1.1$, all these relative errors exhibit slight increases inferior to 1.5\%.} 
\label{table.zeta: 1.5}
\end{table}
\begin{table}[h]
\centering
%\begin{tabular}{|c|c|c|c|c|c|}
\begin{tabular*}{\columnwidth}{@{\extracolsep{\fill}}|c|c|c|c|c|c|}
\hline
N &500&1000 &2,500&5,000&10,000\\
\hline
$\sigma(\hat \kappa_N)$& 26 \% & 18 \% & 11 \% & 8 \% & 6 \%\\
\hline
$\sigma(\mathcal{K}_N)$ & 29 \% & 20 \% & 12 \% & 8 \% & 6 \%\\
\hline
$\sigma(\theta_N)$ & 9 \% & 7 \% & 4 \% & 3 \% & 2 \%\\
\hline
$\sigma(\hat \gamma^2_N)$ & 9 \% & 7 \% & 6 \% & 6 \% & 6 \%\\
\hline
$\sigma(\mathcal{G}^2_N)$ & 7 \% & 5 \% & 3 \% & 2 \% & 2 \%\\
\hline
\end{tabular*}
\caption{ Relative root mean squared errors of volatility parameters estimators for $\zeta= 3.5$. When $\zeta$ decreases to $\zeta =2.5$, these relative errors remain practically unchanged} 
\label{table.zeta: 3.5}
\end{table}
\section{Conclusion}
We have completed a thorough analysis of approximate maximum likelihood parameter estimators for the widely used Heston model \cite{heston1993closed}, where both asset price and squared volatility are jointly driven by a pair of SDEs with non constant coefficients depending on five parameters. Volatilities are not directly observed in practice, but estimated by various well studied techniques. Nevertheless, to gain in clarity, we have focused our theoretical study of parameter estimators on the ideal case where volatilities are jointly observed with stock prices, at $N$ times $T, 2T, \ldots, NT$. The sub-sampling time step $T $ between successive observations is fixed, but can be known or unknown, and we have studied both situations.\\
We have derived explicit closed form expressions of approximate Maximum Likelihood Estimators (MLEs) $\hat \kappa_N, \hat \theta_N, \hat \gamma^2_N$ for the parameters $\kappa, \theta, \gamma^2 $ of the Heston volatility SDE. These formulas enable very fast numerical computations of all estimators. 
For $T$ fixed and $N$ tending to $\infty$, we have computed explicitly the asymptotic bias of our approximate MLEs, and explicitly identified the two key canonical parameters $0 < \omega = \exp{- \kappa T} < 1$ and $\zeta = \kappa \theta / \gamma^2 >1/2$ which control the asymptotic distributions of $\hat \kappa_N, \hat \theta_N, \hat \gamma^2_N$. We show how space and time rescaling reduce the study of these asymptotic distributions to the canonical cases where $\kappa = \gamma = 1$ and $\theta = \zeta$ , with subsampling at time intervals $\kappa T$. We have also constructed explicitly three asymptotically consistent estimators $\mathcal{K}_N, \hat \theta_N, \mathcal{G}^2_N$ of $\kappa, \theta, \gamma^2 $. We have characterized the dichotomy between the case $\zeta > 1$, where all our parameter estimators are asymptotically gaussian, and the case $\zeta < 1$ where their asymptotic distributions have heavy tails similar to those of stable distributions.\\
We have evaluated the small sample accuracy and the concrete speed of convergence of our parameter estimators by intensive simulations of fourcanonical Heston SDEs corresponding to realistic parameter sets. These parameter sets were selected after fitting Heston SDEs to two sets of market data: joint daily observations of the S\&P 500 index and its approximate volatility (the VIX index), joint minute by minute intra-day observations of the Credit Agricole stock price and its estimated Garman-Klass volatility. \\
In a companion paper \cite{azencott1option}, we present practical applications of our parameter estimators and previous results to quantify the sensitivity of estimated option prices to the unavoidable inaccuracy of the Heston SDEs fitted to the underlying asset price and squared volatility. \\
We have also currently exploring in another paper how our asymptotic results extend to situations where the true volatility data are not directly available but are estimated by classical ``realized volatilities" derived from observed stock prices.
\clearpage
\bibliography{refsYG}
\end{document}